\documentclass{article}
\usepackage{makeidx}
\usepackage{graphicx}
\usepackage[english,activeacute]{babel}
\usepackage[latin1]{inputenc}
\usepackage{amsmath}
\usepackage{amsfonts}
\usepackage{amssymb,latexsym}
\usepackage[psamsfonts]{eucal}

\newtheorem{teo}{Theorem}

\newtheorem{lema}{Lemma}

\newenvironment{proof}[1][Proof]{\noindent\textbf{#1.} }{\ \rule{0.5em}{0.5em}}
\begin{document}

\title{The symmetrization problem for multiple orthogonal polynomials}
\author{Am\'ilcar Branquinho$^{1}$ and Edmundo J. Huertas$^{1}$\thanks{%
The work of the first author was partially supported by Centro de
Matem\'atica da Universidade de Coimbra (CMUC), funded by the European
Regional Development Fund through the program COMPETE and by the Portuguese
Government through the FCT - Funda\c c\~ao para a Ci\^encia e a Tecnologia
under the project PEst-C/MAT/UI0324/2011. The work of the second author was
supported by Funda\c c\~ao para a Ci\^encia e a Tecnologia (FCT) of
Portugal, ref. SFRH/BPD/91841/2012, and partially supported by Direcci\'on
General de Investigaci\'on Cient\'ifica y T\'ecnica, Ministerio de
Econom\'ia y Competitividad of Spain, grant MTM2012-36732-C03-01.} \\
$^{1}$Universidade de Coimbra, FCTUC, CMUC and Departamento de Matem\'atica\\
Apartado 3008, 3000 Coimbra, Portugal\\
ajplb@mat.uc.pt, ehuertas@mat.uc.pt}
\date{\emph{(\today)}}
\maketitle

\begin{abstract}
We analyze the effect of symmetrization in the theory of multiple orthogonal
polynomials. For a symmetric sequence of type~II multiple orthogonal
polynomials satisfying a high--term recurrence relation, we fully
characterize the Weyl function associated to the corresponding block Jacobi
matrix as well as the Stieltjes matrix function. Next, from an arbitrary
sequence of type II multiple orthogonal polynomials with respect to a set of 
$d$ linear functionals, we obtain a total of $d+1$ sequences of type II
multiple orthogonal polynomials, which can be used to construct a new
sequence of symmetric type~II multiple orthogonal polynomials. Finally, we
prove a Favard-type result for certain sequences of matrix multiple
orthogonal polynomials satisfying a matrix four--term recurrence relation
with matrix coefficients.
\end{abstract}

AMS SUBJECT CLASSIFICATION (2000): Primary 33C45; Secondary 39B42.


\section{Introduction}

\label{[Section-1]-Intro}


In recent years an increasing attention has been paid to the notion of
multiple orthogonality. Multiple orthogonal polynomials are a generalization
of orthogonal polynomials \cite{Chi78}, satisfying orthogonality conditions
with respect to a number of measures, instead of just one measure. There
exists a vast literature on this subject, e.g. the classical works \cite%
{A-JCAM-98}, \cite{ABV-TAM-03}, \cite[Ch.\thinspace 23]{I-EMA-05} and \cite%
{NS-TrnAMS-91} among others. A characterization through a vectorial
functional equation, where the authors call them $d$\textit{--orthogonal
polynomials} instead of multiple orthogonal polynomials, was done in \cite%
{DM-JAT-82}. Their asymptotic behavior have been studied in \cite{AKS-CA-09}%
, also continued in \cite{DL-CA-2012}, and properties for their zeros have
been analyzed in \cite{HV-JMAA-12}.

Bäcklund transformations resulted from the symmetrization process in the
usual (standard) orthogonality, which allows one to jump, from one hierarchy
to another, in the whole Toda lattices hierarchy (see \cite{GGKM-PRL67}).
That is, they allow reinterpretations inside the hierarchy.\ In \cite%
{BBF-JMAA-13}, the authors have found certain Bäcklund--type transformations
(also known as Miura--type transformations) which allow to reduce problems
in a given full Konstant--Toda hierarchy to another. Also, in\ \cite%
{BB-JDEA-09}, where Peherstorfer's work \cite{P-JCAM-01} is extended to the
whole Toda hierarchy, it is shown how this system can be described with the
evolution of only one parameter instead of two, using exactly this kind of
transformations. Other application to the Toda systems appear in \cite%
{AKI-CA-00}, \cite{BBF-JMAA-10}, and \cite{BBF-JMAA-11}, where the authors
studied Bogoyavlenskii systems which were modeled by certain symmetric
multiple orthogonal polynomials.

In this paper, we are interested in analyze the effect of symmetrization in
systems of multiple orthogonality measures. Our viewpoint seeds some new
light on the subject, and we prove that the symmetrization process in
multiple orthogonality is a model to define the aforementioned Bä%
cklund--type transformations, as happens in the scalar case with the Bä%
cklund transformations (see \cite{Chi78}, \cite{MS-NA-92}, \cite{P-JCAM-01}%
). Furthermore, we solve the so called \textit{symmetrization problem} in
the theory of multiple orthogonal polynomials. We apply certain \textit{%
Darboux transformations}, already described in \cite{BBF-JMAA-13}, to a $%
(d+2)$--banded matrix, associated to a $(d+2)$--term recurrence relation
satisfied by an arbitrary sequence of type~II multiple orthogonal
polynomials, to obtain a total of $d+1$ sequences of not necessarily
symmetric multiple orthogonal polynomials, which we use to construct a new
sequence of symmetric multiple orthogonal polynomials.

On the other hand, following the ideas in \cite{MS-NA-92} (and the
references therein) for standard sequences of orthogonal polynomials, in 
\cite{MMR-JDEA-11} (see also\ \cite{DM-A-92}) the authors provide a cubic
decomposition for sequences of polynomials, multiple orthogonal with respect
to a two different linear functionals. Concerning the symmetric case, in 
\cite{MM-MJM13} this cubic decomposition is analyzed for a 2-symmetric
sequence of polynomials, which is called a \textit{diagonal cubic
decomposition (CD)} by the authors. Here, we also extend this notion of
diagonal decomposition to a more general case, considering symmetric
sequences of polynomials multiple orthogonal with respect to $d>3$ linear
functionals.

The structure of the manuscript is as follows. In Section \ref%
{[Section-2]-Defs} we summarize without proofs the relevant material about
multiple orthogonal polynomials, and a basic background about the matrix
interpretation of the type~II multi-orthogonality conditions with respect to
the a regular system of $d$ linear functionals $\{u^{1},\ldots ,u^{d}\}$ and
diagonal multi--indices. In Section \ref{[Section-3]-FyR-Functions} we fully
characterize the Weyl function $\mathcal{R}_{J}$ and the Stieltjes matrix
function $\mathcal{F}$ associated to the block Jacobi matrix $J$
corresponding to a $(d+2)$--term recurrence relation satisfied by a
symmetric sequence of type~II multiple orthogonal polynomials. In Section %
\ref{[Section-4]-SymProb}, starting from an arbitrary sequence of type~II
multiple polynomials satisfying a $(d+2)$--term recurrence relation, we
state the conditions to find a total of $d+1$ sequences of type~II multiple
orthogonal polynomials, in general non--symmetric, which can be used to
construct a new sequence of \textit{symmetric} type~II multiple orthogonal
polynomials. Moreover, we also deal with the converse problem, i.e., we
propose a decomposition of a given \textit{symmetric} type~II multiple
orthogonal polynomial sequence, which allows us to find a set of other (in
general non--symmetric) $d+1$ sequences of type~II multiple orthogonal
polynomials, satisfying in turn $(d+2)$--term recurrence relations. Finally,
in Section \ref{[Section-5]-FavardTh}, we present a Favard-type result,
showing that certain $3\times 3$ matrix decomposition of a type~II multiple $%
2$--orthogonal polynomials, satisfy a \textit{matrix four--term recurrence
relation}, and therefore it is type~II multiple $2$--orthogonal (in a matrix
sense) with respect to a certain system of matrix measures.


\section{Definitions and matrix interpretation of multiple orthogonality}

\label{[Section-2]-Defs}


Let $\boldsymbol{n}=(n_{1},...,n_{d})\in \mathbb{N}^{d}$ be a multi--index
with length $|\mathbf{n}|:=n_{1}+\cdots +n_{d}$ and let $\{u^{j}\}_{j=1}^{d}$
be a set of linear functionals, i.e. $u^{j}:\mathbb{P}\rightarrow \mathbb{C}$%
. Let $\{P_{\mathbf{n}}\}$ be a sequence of polynomials, with $\deg P_{%
\mathbf{n}}$ is at most $|\mathbf{n}|$. $\{P_{\mathbf{n}}\}$ is said to be 
\textit{type~II multiple orthogonal} with respect to the set of linear
functionals $\{u_{j}\}_{j=1}^{d}$ and multi--index $\boldsymbol{n}$ if%
\begin{equation}
u^{j}(x^{k}P_{\mathbf{n}})=0\,,\ \ k=0,1,\ldots ,n_{j}-1\,,\ \ j=1,\ldots
,d\,.  \label{[Sec2]-OrtConduj}
\end{equation}%
A multi--index $\boldsymbol{n}$ is said to be \textit{normal} for the set of
linear functionals $\{u_{j}\}_{j=1}^{d}$, if the degree of $P_{\mathbf{n}}$
is exactly $|\mathbf{n}|=n$. When all the multi--indices of a given family
are normal, we say that the set of linear functionals $\{u_{j}\}_{j=1}^{d}$
is \textit{regular}. In the present work, we will restrict our attention
ourselves to the so called \textit{diagonal multi--indices} $\boldsymbol{n}%
=(n_{1},...,n_{d})\in \mathcal{I}$, where%
\begin{equation*}
\mathcal{I}=\{(0,0,\ldots ,0),(1,0,\ldots ,0),\ldots ,(1,1,\ldots
,1),(2,1,\ldots ,1),\ldots ,(2,2,\ldots ,2),\ldots \}.
\end{equation*}%
Notice that there exists a one to one correspondence, $\mathbf{i}$, between
the above set of diagonal multi--indices $\mathcal{I}\subset \mathbb{N}^{d}$
and $\mathbb{N}$, given by $\mathbf{i}(\mathbb{N}^{d})=|\mathbf{n}|=n$.
Therefore, to simplify the notation, we write in the sequel $P_{\mathbf{n}%
}\equiv P_{|\mathbf{n}|}=P_{n}$. The left--multiplication of a linear
functional $u:\mathbb{P}\rightarrow \mathbb{C}$ by a polynomial $p\in 
\mathbb{P}$ is given by the new linear functional $p\,u:\mathbb{P}%
\rightarrow \mathbb{C}$ such that%
\begin{equation*}
p\,u(x^{k})=u(p(x)x^{k})\,,\ \ k\in \mathbb{N}\,.
\end{equation*}

Next, we briefly review a matrix interpretation of type~II multiple
orthogonal polynomials with respect to a system of $d$ regular linear
functionals and a family of diagonal multi--indices. Throughout this work,
we will use this matrix interpretation as a useful tool to obtain some of
the main results of the manuscript. For a recent and deeper account of the
theory (in a more general framework, considering quasi--diagonal
multi--indices) we refer the reader to \cite{BCF-NA-10}.

Let us consider the family of vector polynomials%
\begin{equation*}
\mathbb{P}^{d}=\{%
\begin{bmatrix}
P_{1} & \cdots & P_{d}%
\end{bmatrix}%
^{T},\ \ d\in \mathbb{N},\,P_{j}\in \mathbb{P}\},
\end{equation*}%
and $\mathcal{M}_{d\times d}$ the set of $d\times d$ matrices with entries
in $\mathbb{C}$. Let $\{\mathcal{X}_{j}\}$ be the family of vector
polynomials $\mathcal{X}_{j}\in \mathbb{P}^{d}$\ defined by%
\begin{equation}
\mathcal{X}_{j}=%
\begin{bmatrix}
x^{jd} & \cdots & x^{(j+1)d-1}%
\end{bmatrix}%
^{T},\ \ j\in \mathbb{N},  \label{[Sec2]-vecXj}
\end{equation}%
where $\mathcal{X}_{0}=%
\begin{bmatrix}
1 & \cdots & x^{d-1}%
\end{bmatrix}%
^{T}$. By means of the shift $n\rightarrow nd$, associated with $\{P_{n}\}$,
we define the sequence of vector polynomials $\{\mathcal{P}_{n}\}$, with%
\begin{equation}
\mathcal{P}_{n}=%
\begin{bmatrix}
P_{nd}(x) & \cdots & P_{(n+1)d-1}(x)%
\end{bmatrix}%
^{T},\ \ n\in \mathbb{N},\,\mathcal{P}_{n}\in \mathbb{P}^{d}.
\label{[Sec2]-vecPn}
\end{equation}%
Let $u^{j}:\mathbb{P}\rightarrow \mathbb{C}$ with $j=1,\ldots ,d$ a system
of linear functionals as in (\ref{[Sec2]-OrtConduj}). From now on, we define
the \textit{vector of functionals }$\mathcal{U}=%
\begin{bmatrix}
u^{1} & \cdots & u^{d}%
\end{bmatrix}%
^{T}$ acting in $\mathbb{P}^{d}\rightarrow \mathcal{M}_{d\times d}$, by%
\begin{equation*}
\mathcal{U}(\mathcal{P})=\left( \mathcal{U}\text{\textperiodcentered }%
\mathcal{P}^{T}\right) ^{T}=%
\begin{bmatrix}
u^{1}(P_{1}) & \cdots & u^{d}(P_{1}) \\ 
\vdots & \ddots & \vdots \\ 
u^{1}(P_{d}) & \cdots & u^{d}(P_{d})%
\end{bmatrix}%
.
\end{equation*}%
Let%
\begin{equation*}
A_{\ell }(x)=\sum_{k=0}^{\ell }A_{k}^{\ell }\,x^{k},
\end{equation*}%
be a matrix polynomial of degree $\ell $, where $A_{k}^{\ell }\in \mathcal{M}%
_{2\times 2}$, and $\mathcal{U}$ a vector of functional. We define the new
vector of functionals called \textit{left multiplication of }$\mathcal{U}$%
\textit{\ by a matrix polynomial }$A_{\ell }$, and we denote it by $A_{\ell }%
\mathcal{U}$,\ to the map of~$\mathbb{P}^{d}$ into $\mathcal{M}_{d\times d}$%
, described by%
\begin{equation}
\left( A_{\ell }\mathcal{U}\right) (\mathcal{P})=\sum_{k=0}^{\ell }\left(
x^{k}\mathcal{U}\right) \left( \mathcal{P}\right) (A_{k}^{n})^{T}.
\label{[Sec2]-Def-LeftMultp}
\end{equation}%
From (\ref{[Sec2]-Def-LeftMultp}) we introduce the notion of \textit{moments
of order }$j\in \mathbb{N}$, associated with the vector of functionals $x^{k}%
\mathcal{U}$, which will be in general the following $d\times d$\ matrices%
\begin{equation*}
\mathcal{U}_{j}^{k}=\left( x^{k}\mathcal{U}\right) (\mathcal{X}_{j})=%
\begin{bmatrix}
u^{1}(x^{jd+k}) & \cdots & u^{d}(x^{jd+k}) \\ 
\vdots & \ddots & \vdots \\ 
u^{1}(x^{(j+1)d-1+k}) & \cdots & u^{d}(x^{(j+1)d-1+k})%
\end{bmatrix}%
,
\end{equation*}%
with $j,k\in \mathbb{N}$, and from this moments, we construct the \textit{%
block \textit{Hankel} matrix of moments}%
\begin{equation*}
\mathcal{H}_{n}=%
\begin{bmatrix}
\mathcal{U}_{0}^{0} & \cdots & \mathcal{U}_{0}^{n} \\ 
\vdots & \ddots & \vdots \\ 
\mathcal{U}_{n}^{0} & \cdots & \mathcal{U}_{n}^{n}%
\end{bmatrix}%
\,,\ \ n\in \mathbb{N}.
\end{equation*}%
We say that the vector of functionals $\mathcal{U}$ is \textit{regular}, if
the determinants of the principal minors of the above matrix are non-zero
for every $n\in \mathbb{N}$. Having in mind (\ref{[Sec2]-vecXj}) it is
obvious that $\mathcal{X}_{j}=(x^{d})^{j}\mathcal{X}_{0},\,\,\,j\in \mathbb{N%
}$. Thus, from (\ref{[Sec2]-vecPn}) we can express $\mathcal{P}_{n}(x)$ in
the alternative way%
\begin{equation}
\mathcal{P}_{n}(x)=\sum_{j=0}^{n}P_{j}^{n}\mathcal{X}_{j}\,,\ \ P_{j}^{n}\in 
\mathcal{M}_{d\times d}\,,  \label{[Sec2]-MatrForm-I}
\end{equation}%
where the matrix coefficients $P_{j}^{n},\ j=0,1,\ \ldots ,\ n$ are uniquely
determined. Thus, it also occurs%
\begin{equation}
\mathcal{P}_{n}(x)=W_{n}(x^{d})\mathcal{X}_{0}\,,  \label{[Sec2]-MatForm-II}
\end{equation}%
where $W_{n}$ is a matrix polynomial (i.e., $W_{n}$ is a $d\times d$ matrix
whose entries are polynomials) of degree $n$ and dimension $d$, given by%
\begin{equation}
W_{n}(x)=\sum_{j=0}^{n}P_{j}^{n}x^{j}\,,\ \ P_{j}^{n}\in \mathcal{M}%
_{d\times d}.  \label{[Sec2]-MatForm-II-Vn}
\end{equation}%
Notice that the matrices $P_{j}^{n}\in \mathcal{M}_{d\times d}$ in (\ref%
{[Sec2]-MatForm-II-Vn}) are the same as in (\ref{[Sec2]-MatrForm-I}). Within
this context, we can now describe the matrix interpretation of multiple
orthogonality for diagonal multi--indices. Let $\{\mathcal{P}_{n}\}$ be a
sequence of vector polynomials with polynomial entries as in (\ref%
{[Sec2]-vecPn}), and a vector of functionals $\mathcal{U}$ as described
above. $\{\mathcal{P}_{n}\}$ is said to be a \textit{type~II vector multiple
orthogonal polynomial sequence} with respect to the vector of functionals $%
\mathcal{U}$, and a set of diagonal multi--indices, if%
\begin{equation}
\left. 
\begin{array}{rll}
i) & (x^{k}\mathcal{U})(\mathcal{P}_{n})=0_{d\times d}\,, & k=0,1,\ldots
,n-1\,, \\ 
ii) & (x^{n}\mathcal{U})(\mathcal{P}_{n})=\Delta _{n}\,, & 
\end{array}%
\right\}  \label{[Sec2]-OrthCond-U}
\end{equation}%
where $\Delta _{n}$ is a regular upper triangular $d\times d$ matrix (see 
\cite[Th. 3]{BCF-NA-10} considering diagonal multi--indices).

Next, we introduce a few aspects of the duality theory, which will be useful
in the sequel. We denote by $\mathbb{P}^{\ast }$ the dual space of $\mathbb{P%
}$, i.e. the linear space of linear functionals defined on $\mathbb{P}$ over 
$\mathbb{C}$. Let $\{P_{n}\}$ be a sequence of monic polynomials. We call $%
\{\ell _{n}\}$, $\ell _{n}\in \mathbb{P}^{\ast }$, the \textit{dual sequence}%
\ of $\{P_{n}\}$\ if $\ell _{i}(P_{j})=\delta _{i,j},\,\,i,\,j\in \mathbb{N}$%
\ holds. Given a sequence of linear functionals $\{\ell _{n}\}\in \mathbb{P}%
^{\ast }$, by means of the shift $n\rightarrow nd$, the vector sequence of
linear functionals $\{\mathcal{L}_{n}\}$, with%
\begin{equation}
\mathcal{L}_{n}=%
\begin{bmatrix}
\ell _{nd} & \cdots & \ell _{(n+1)d-1}%
\end{bmatrix}%
^{T},\ \ n\in \mathbb{N},  \label{[Sec2]-vecLinFrm}
\end{equation}%
is said to be the \textit{vector sequence of linear functionals} associated
with $\{\ell _{n}\}$.

It is very well known (see \cite{DM-JAT-82}) that a given sequence of
type~II polynomials $\{P_{n}\}$, simultaneously orthogonal with respect to a 
$d$ linear functionals, or simply $d$\textit{--orthogonal} polynomials,
satisfy the following $(d+2)$--term order recurrence relation%
\begin{equation}
xP_{n+d}(x)=P_{n+d+1}(x)+\beta _{n+d}P_{n+d}(x)+\sum_{\nu =0}^{d-1}\gamma
_{n+d-\nu }^{d-1-\nu }P_{n+d-1-\nu }(x)\,,  \label{[Sec2]-HTRR-MultOP}
\end{equation}%
$\gamma _{n+1}^{0}\neq 0$ for $n\geq 0$, with the initial conditions $%
P_{0}(x)=1$, $P_{1}(x)=x-\beta _{0}$, and%
\begin{equation*}
P_{n}(x)=(x-\beta _{n-1})P_{n-1}(x)-\sum_{\nu =0}^{n-2}\gamma _{n-1-\nu
}^{d-1-\nu }P_{n-2-\nu }(x)\,,\ \ 2\leq n\leq d\,.
\end{equation*}%
E.g., if $d=2$, the sequence of monic type~II multiple orthogonal
polynomials $\{P_{n}\}$ with respect to the regular system of functionals $%
\{u^{1},u^{2}\}$ and normal multi--index satisfy, for every $n\geq 0$, the
following four term recurrence relation (see \cite[Lemma 1-a]{BCF-NA-10}, 
\cite{K-JCAM-95})%
\begin{equation}
xP_{n+2}(x)=P_{n+3}(x)+\beta _{n+2}P_{n+2}(x)+\gamma
_{n+2}^{1}P_{n+1}(x)+\gamma _{n+1}^{0}P_{n}(x)\,,  \label{[Sec2]-4TRR-P}
\end{equation}%
where $\beta _{n+2},\gamma _{n+2}^{1},\gamma _{n+1}^{0}\in \mathbb{C}$, $%
\gamma _{n+1}^{0}\neq 0$, $P_{0}(x)=1$, $P_{1}(x)=x-\beta _{0}$ and $%
P_{2}(x)=(x-\beta _{1})P_{1}(x)-\gamma _{1}^{1}P_{0}(x)$.

We follow \cite[Def. 4.1.]{DM-JAT-82} in assuming that a monic system of
polynomials $\{S_{n}\}$ is said to be $d$\textit{--symmetric} when it
verifies%
\begin{equation}
S_{n}(\xi _{k}x)=\xi _{k}^{n}S_{n}(x)\,,\ \ n\geq 0\,,
\label{[Sec2]-d-symmetric-Sn}
\end{equation}%
where $\xi _{k}=\exp \left( {2k\pi i}/({d+1})\right) $, $k=1,\ldots ,d$, and 
$\xi _{k}^{d+1}=1$. Notice that, if $d=1$, then $\xi _{k}=-1$ and therefore $%
S_{n}(-x)=(-1)^{n}S_{n}(x)$ (see~\cite{Chi78}). We also assume (see \cite[%
Def. 4.2.]{DM-JAT-82}) that the vector of linear functionals $\mathcal{L}%
_{0}=%
\begin{bmatrix}
\ell _{0} & \cdots & \ell _{d-1}%
\end{bmatrix}%
^{T}$ is said to be $d$\textit{--symmetric} when the moments of its entries
satisfy, for every $n\geq 0$,%
\begin{equation}
\ell _{\nu }(x^{(d+1)n+\mu })=0\,,\ \ \nu =0,1,\ldots ,d-1\,,\ \ \mu
=0,1,\ldots ,d\,,\ \ \nu \neq \mu \,.  \label{[Sec2]-momL1}
\end{equation}%
Observe that if $d=1$, this condition leads to the well known fact $\ell
_{0}(x^{2n+1})=0$, i.e., all the odd moments of a symmetric moment
functional are zero (see~\cite[Def. 4.1, p.20]{Chi78}).

Under the above assumptions, we have the following

\begin{teo}[{cf. \protect\cite[Th. 4.1]{DM-JAT-82}}]
\label{[SEC2]-TH41-DMJAT82} For every sequence of monic polynomials $%
\{S_{n}\}$, $d$--orthogonal with respect to the \textit{vector of linear
functionals }$\mathcal{L}_{0}=%
\begin{bmatrix}
\ell _{0} & \cdots & \ell _{d-1}%
\end{bmatrix}%
^{T}$, the following statements are equivalent:

\begin{itemize}
\item[$(a)$] The vector of linear functionals $\mathcal{L}_{0}$ is $d$%
--symmetric.

\item[$(b)$] The sequence $\{S_{n}\}$ is $d$--symmetric.

\item[$(c)$] The sequence $\{S_{n}\}$ satisfies%
\begin{equation}
xS_{n+d}(x)=S_{n+d+1}(x)+\gamma _{n+1}S_{n}(x)\,,\ \ n\geq 0,
\label{[Sec2]-HTRR-Symm}
\end{equation}%
with $S_{n}(x)=x^{n}$ for $0\leq n\leq d$.
\end{itemize}
\end{teo}

Notice that (\ref{[Sec2]-HTRR-Symm}) is a particular case of\ the $(d+2)$%
--term recurrence relation (\ref{[Sec2]-HTRR-MultOP}). Continuing the same
trivial example above for $d=2$, it directly implies that the sequence of
polynomials $\{S_{n}\}$, satisfy the particular case of~(\ref%
{[Sec2]-HTRR-MultOP}), i.e. $S_{0}(x)=1$, $S_{1}(x)=1$, $S_{2}(x)=x^{2}$ and%
\begin{equation*}
xS_{n+2}(x)=S_{n+3}(x)+\gamma _{n+1}S_{n}(x)\,,\ \ n\geq 0.
\end{equation*}%
Notice that the coefficients $\beta _{n+2}$ and $\gamma _{n+2}^{1}$ of
polynomials $S_{n+2}$ and $S_{n+1}$ respectively, on the right hand side of (%
\ref{[Sec2]-4TRR-P}) are zero.

On the other hand, the $(d+2)$--term recurrence relation (\ref%
{[Sec2]-HTRR-Symm}) can be rewritten in terms of vector polynomials (\ref%
{[Sec2]-vecPn}), and then we obtain what will be referred to as the \textit{%
symmetric type~II vector multiple orthogonal polynomial sequence} $\mathcal{S%
}_{n}=%
\begin{bmatrix}
S_{nd} & \cdots & S_{(n+1)d-1}%
\end{bmatrix}%
^{T}$. For $n\rightarrow dn+j$, $j=0,1,\ldots d-1$ and $n\in \mathbb{N}$, we
have the following matrix three term recurrence relation%
\begin{equation}
x\mathcal{S}_{n}=A\mathcal{S}_{n+1}+B\mathcal{S}_{n}+C_{n}\mathcal{S}%
_{n-1}\,,\ \ n=0,1,\ldots  \label{[Sec2]-3TRR-S}
\end{equation}%
with $\mathcal{S}_{-1}=%
\begin{bmatrix}
0 & \cdots & 0%
\end{bmatrix}%
^{T}$, $\mathcal{S}_{0}=%
\begin{bmatrix}
S_{0} & \cdots & S_{d-1}%
\end{bmatrix}%
^{T}$, and matrix coefficients $A$, $B$, $C_{n}\in \mathcal{M}_{d\times d}$
given%
\begin{eqnarray}
A &=&%
\begin{bmatrix}
0 & 0 & \cdots & 0 \\ 
\vdots & \vdots & \ddots & \vdots \\ 
0 & 0 & \cdots & 0 \\ 
1 & 0 & \cdots & 0%
\end{bmatrix}%
\,,\ B=%
\begin{bmatrix}
0 & 1 &  &  \\ 
& \ddots & \ddots &  \\ 
&  & 0 & 1 \\ 
&  &  & 0%
\end{bmatrix}%
,\ \ \mbox{ and }  \label{[Sec2]-3TTR-MCoefs} \\
C_{n} &=&\operatorname{diag}\,[\gamma _{1},\,\gamma _{2},\ldots ,\gamma _{nd}]. 
\notag
\end{eqnarray}%
Note that, in this case, one has%
\begin{equation*}
\mathcal{S}_{0}=\mathcal{X}_{0}=%
\begin{bmatrix}
1 & x & \cdots & x^{d-1}%
\end{bmatrix}%
^{T}.
\end{equation*}%
Since $\{S_{n}\}$\ satisfies (\ref{[Sec2]-HTRR-Symm}), it is clear that this 
$d$--symmetric type~II multiple polynomial sequence will be orthogonal with
respect to certain system of $d$ linear functionals, say $\{v^{1},\ldots
,v^{d}\}$. Hence, according to the matrix interpretation of multiple
orthogonality, the corresponding type II vector multiple polynomial sequence 
$\{\mathcal{S}_{n}\}$ will be orthogonal with respect to a \textit{symmetric
vector of functionals }$\mathcal{V}=%
\begin{bmatrix}
v^{1} & \cdots & v^{d}%
\end{bmatrix}%
^{T}$. The corresponding matrix orthogonality conditions for $\{\mathcal{S}%
_{n}\}$ and $\mathcal{V}$ are described in (\ref{[Sec2]-OrthCond-U}).

One of the main goals of this manuscript is to analyze symmetric sequences
of type~II vector multiple polynomials, orthogonal with respect to a
symmetric vector of functionals. The remainder of this section will be
devoted to the proof of one of our main results concerning the moments of the%
$\ d$ functional entries of such symmetric vector of functionals $\mathcal{V}
$. The following theorem states that, under certain conditions, the moment
of each functional entry in $\mathcal{V}$ can be given in terms of the
moments of other functional entry in the same $\mathcal{V}$.

\begin{lema}
\label{[SEC2]-TH2-MOM-v1v2} If $\mathcal{V}=%
\begin{bmatrix}
v^{1} & \cdots & v^{d}%
\end{bmatrix}%
^{T}$ is a symmetric vector of functionals, the moments of each functional
entry $v^{j}$, $j=1,2,\ldots ,d$ in $\mathcal{V}$, can be expressed for all $%
n\geq 0$ as

\begin{itemize}
\item[$(i)$] If $\mu =0,1,\ldots ,j-2$%
\begin{equation}
v^{j}(x^{(d+1)n+\mu })=\frac{v_{j,\mu }}{v_{\mu +1,\mu }}v^{\mu
+1}(x^{(d+1)n+\mu }),  \label{[Sec2]-momVsym}
\end{equation}%
where $v_{k,l}=v^{k}(S_{l})$.

\item[$(ii)$] If $\mu =j-1$, the value $v^{j}(x^{(d+1)n+\mu })$ depends on $%
v^{j}$, and it is different from zero.

\item[$(iii)$] If $\mu =j,j+1,\ldots ,d$%
\begin{equation*}
v^{j}(x^{(d+1)n+\mu })=0.
\end{equation*}
\end{itemize}
\end{lema}

\begin{proof}
In the matrix framework of multiple orthogonality, the type~II vector
polynomials $\mathcal{S}_{n}$ are multiple orthogonal with respect to the
symmetric vector moment of functionals $\mathcal{V}:\mathbb{P}%
^{d}\rightarrow \mathcal{M}_{d\times d}$, with $\mathcal{V}=%
\begin{bmatrix}
v^{1} & \cdots & v^{d}%
\end{bmatrix}%
^{T}$. If we multiply (\ref{[Sec2]-3TRR-S}) by~$x^{n-1}$, together with the
linearity of $\mathcal{V}$, we get, for $n=0,1,\ldots $ 
\begin{equation*}
\mathcal{V}\left( x^{n}\mathcal{S}_{n}\right) =A\mathcal{V}\left( x^{n-1}%
\mathcal{S}_{n+1}\right) +B\mathcal{V}\left( x^{n-1}\mathcal{S}_{n}\right)
+C_{n}\mathcal{V}\left( x^{n-1}\mathcal{S}_{n-1}\right) \,,\,\,n=0,1,\ldots .
\end{equation*}%
By the orthogonality conditions (\ref{[Sec2]-OrthCond-U}) for $\mathcal{V}$,
we have%
\begin{equation*}
A\mathcal{V}\left( x^{n-1}\mathcal{S}_{n+1}\right) =B\mathcal{V}\left(
x^{n-1}\mathcal{S}_{n}\right) =0_{d\times d}\,,\ \ n=0,1,\ldots \,,
\end{equation*}%
and iterating the remain expression%
\begin{equation*}
\mathcal{V}\left( x^{n}\mathcal{S}_{n}\right) =C_{n}\mathcal{V}\left( x^{n-1}%
\mathcal{S}_{n-1}\right) \,,\ \ n=0,1,\ldots
\end{equation*}%
we obtain%
\begin{equation*}
\mathcal{V}\left( x^{n}\mathcal{S}_{n}\right) =C_{n}C_{n-1}\cdots C_{1}%
\mathcal{V}\left( \mathcal{S}_{0}\right) \,,\ \ n=0,1,\ldots .
\end{equation*}%
The above matrix $\mathcal{V}\left( \mathcal{S}_{0}\right) $ is given by%
\begin{equation*}
\mathcal{V}\left( \mathcal{S}_{0}\right) =\mathcal{V}_{0}^{0}=%
\begin{bmatrix}
v^{1}(S_{0}) & \cdots & v^{d}(S_{0}) \\ 
\vdots & \ddots & \vdots \\ 
v^{1}(S_{d-1}) & \cdots & v^{d}(S_{d-1})%
\end{bmatrix}%
.
\end{equation*}%
To simplify the notation, in the sequel $v_{i,j-1}$ denotes $v^{i}(S_{j-1})$%
. Notice that (\ref{[Sec2]-OrthCond-U})\ leads to the fact that the above
matrix is an upper triangular matrix, which in turn means that $v_{i,j-1}=0$%
, for every $i,j=1,\ldots ,d-1$, that is%
\begin{equation}
\mathcal{V}_{0}^{0}=%
\begin{bmatrix}
v_{1,0} & v_{2,0} & \cdots & v_{d,0} \\ 
& v_{2,1} & \cdots & v_{d,1} \\ 
&  & \ddots & \vdots \\ 
&  &  & v_{d,d-1}%
\end{bmatrix}%
.  \label{[Sec2]-V00}
\end{equation}%
Let $\mathcal{L}_{0}$ be a $d$--symmetric vector of linear functionals as in
Theorem \ref{[SEC2]-TH41-DMJAT82}. We can express $\mathcal{V}$ in terms of $%
\mathcal{L}_{0}$ as $\mathcal{V}=G_{0}\mathcal{L}_{0}$. Thus, we have%
\begin{equation*}
(G_{0}^{-1}\,\mathcal{V})(\mathcal{S}_{0})=\mathcal{L}_{0}(\mathcal{S}%
_{0})=I_{d}\,.
\end{equation*}%
From (\ref{[Sec2]-Def-LeftMultp}) we have%
\begin{equation*}
\left( G_{0}^{-1}\,\mathcal{V}\right) (\mathcal{S}_{0})=\mathcal{V}\left( 
\mathcal{S}_{0}\right) (G_{0}^{-1})^{T}=\mathcal{V}%
_{0}^{0}(G_{0}^{-1})^{T}=I_{d}\,.
\end{equation*}%
Therefore, taking into account (\ref{[Sec2]-V00}), we conclude%
\begin{equation}
G_{0}=(\mathcal{V}_{0}^{0})^{T}=%
\begin{bmatrix}
v_{1,0} &  &  &  \\ 
v_{2,0} & v_{2,1} &  &  \\ 
\vdots & \vdots & \ddots &  \\ 
v_{d,0} & v_{d,1} & \cdots & v_{d,d-1}%
\end{bmatrix}%
.  \label{[Sec2]-G0}
\end{equation}%
Observe that the matrix $(\mathcal{V}_{0}^{0})^{T}$ is lower triangular, and
every entry in their main diagonal is different from zero, so $G_{0}$ always
exists and is a lower triangular matrix. Since $\mathcal{V}=G_{0}\mathcal{L}%
_{0}$, we finally obtain the expressions%
\begin{equation}
\begin{array}{l}
v^{1}=v_{1,0}\ell _{0}, \\ 
v^{2}=v_{2,0}\ell _{0}+v_{2,1}\ell _{1}, \\ 
v^{3}=v_{3,0}\ell _{0}+v_{3,1}\ell _{1}+v_{3,2}\ell _{2}, \\ 
\cdots \\ 
v^{d}=v_{d,0}\ell _{0}+v_{d,1}\ell _{1}+\cdots +v_{d,d-1}\ell _{d-1},%
\end{array}
\label{[Sec2]-RelEntreFors}
\end{equation}%
between the entries of $\mathcal{L}_{0}$\ and $\mathcal{V}$.

Next, from (\ref{[Sec2]-RelEntreFors}), (\ref{[Sec2]-momL1}),\ together with
the crucial fact that every value in the main diagonal of $\mathcal{V}%
_{0}^{0}$ is different from zero, it is a simple matter to check that the
three statements of the lemma follow. We conclude the proof only for the
functionals $v^{2}$ and $v^{1}$. The other cases can be deduced in a similar
way.

From (\ref{[Sec2]-RelEntreFors}) we get $v^{1}(x^{(d+1)n+\mu })=v_{1,0}\cdot
\ell _{0}(x^{(d+1)n+\mu })\neq 0$. Then, from (\ref{[Sec2]-momL1}) we see
that for every $\mu \neq 0$ we have $v^{1}(x^{(d+1)n+\mu })=0$\ (statement $%
(iii)$).\ If $\mu =0$, we have $v^{1}(x^{(d+1)n})=v_{1,0}\cdot \ell
_{0}(x^{(d+1)n})\neq 0$ (statement $(ii)$).\ Next, from (\ref%
{[Sec2]-RelEntreFors}) we get $v^{2}(x^{(d+1)n+\mu })=v_{2,0}\cdot \ell
_{0}(x^{(d+1)n+\mu })+v_{2,1}\cdot \ell _{1}(x^{(d+1)n+\mu })$. Then, from (%
\ref{[Sec2]-momL1}) we see that for every $\mu \neq 1$, we have $%
v^{2}(x^{(d+1)n+\mu })=0$ (statement $(iii)$). If $\mu =0$ we have%
\begin{equation*}
v^{2}(x^{(d+1)n})=\frac{v_{2,0}}{v_{1,0}}v^{1}(x^{(d+1)n})\,\,\,\text{%
(statement }(i)\text{).}
\end{equation*}%
If $\mu =1$ then\ $v^{2}(x^{(d+1)n+1})=v_{2,1}\cdot \ell
_{1}(x^{(d+1)n+1})\neq 0$ (statement $(ii)$).

Thus, the lemma follows.
\end{proof}


\section{Representation of the Stieltjes and Weyl functions}

\label{[Section-3]-FyR-Functions}


Let $\mathcal{U}$ a vector of functionals $\mathcal{U}=%
\begin{bmatrix}
u^{1} & \cdots & u^{d}%
\end{bmatrix}%
^{T}$. We define the \textit{Stieltjes matrix function }associated to $%
\mathcal{U}$ (or \textit{matrix generating function} associated to $\mathcal{%
U}$), $\mathcal{F}$ by (see \cite{BCF-NA-10})%
\begin{equation*}
\mathcal{F}(z)=\sum_{n=0}^{\infty }\frac{\left( x^{n}\mathcal{U}\right) (%
\mathcal{X}_{0}(x))}{z^{n+1}}.
\end{equation*}

In this Section we find the relation between the Stieltjes matrix function $%
\mathcal{F}$, associated to a certain $d$--symmetric vector of functionals $%
\mathcal{V}$, and certain interesting function associated to the
corresponding block Jacobi matrix $J$. Here we deal with $d$--symmetric
sequences of type~II multiple orthogonal polynomials $\{S_{n}\}$, and hence $%
J$ is a $(d+2)$--banded matrix with only two extreme non-zero diagonals,
which is the block--matrix representation of the three--term recurrence
relation with $d\times d$ matrix coefficients, satisfied by the vector
sequence of polynomials $\{\mathcal{S}_{n}\}$ (associated to $\{S_{n}\}$),
orthogonal with respect to $\mathcal{V}$.

Thus, the shape of a Jacobi matrix $J$, associated with the $(d+2)$--term
recurrence relation (\ref{[Sec2]-HTRR-Symm}) satisfied by a $d$--symmetric
sequence of type~II multiple orthogonal polynomials $\{S_{n}\}$ is%
\begin{equation}
J=%
\begin{bmatrix}
0 & 1 &  &  &  &  &  \\ 
0 & 0 & 1 &  &  &  &  \\ 
\vdots &  & \ddots & \ddots &  &  &  \\ 
0 &  &  & 0 & 1 &  &  \\ 
\gamma _{1} & 0 &  &  & 0 & 1 &  \\ 
& \gamma _{2} & 0 &  &  & 0 & \ddots \\ 
&  & \ddots & \ddots &  &  & \ddots%
\end{bmatrix}%
.  \label{[Sec3]-JacMatxJ}
\end{equation}%
We can rewrite $J$ as the block tridiagonal matrix%
\begin{equation*}
J=%
\begin{bmatrix}
B & A &  &  &  \\ 
C_{1} & B & A &  &  \\ 
& C_{2} & B & A &  \\ 
&  & \ddots & \ddots & \ddots%
\end{bmatrix}%
\end{equation*}%
associated to a three term recurrence relation with matrix coefficients,
satisfied by the sequence of type~II vector multiple orthogonal polynomials $%
\{\mathcal{S}_{n}\}$ associated to $\{S_{n}\}$. Here, every block matrix $A$%
, $B$ and $C_{n}$ has $d\times d$ size, and they are given in (\ref%
{[Sec2]-3TTR-MCoefs}). When $J$ is a bounded operator, it is possible to
define the \textit{resolvent operator} by%
\begin{equation*}
(zI-J)^{-1}=\sum_{n=0}^{\infty }\frac{J^{n}}{z^{n+1}},\,\,\,|z|>||J||,
\end{equation*}%
(see \cite{B-JAT-99}) and we can put in correspondence the following block
tridiagonal analytic function, known as the \textit{Weyl function associated
with }$J$%
\begin{equation}
\mathcal{R}_{J}(z)=\sum_{n=0}^{\infty }\frac{\mathbf{e}_{0}^{T}\,J^{n}\,%
\mathbf{e}_{0}}{z^{n+1}},\,\,\,|z|>||J||,  \label{[Sec3]-R-Weyl}
\end{equation}%
where $\mathbf{e}_{0}=%
\begin{bmatrix}
I_{d} & 0_{d\times d} & \cdots%
\end{bmatrix}%
^{T}$. If we denote by $M_{ij}$ the $d\times d$\ block matrices of a
semi-infinite matrix $M$, formed by the entries of rows $d(i-1)+1,d(i-1)+2,%
\ldots ,di\,$, and columns $d(j-1)+1,d(j-1)+2,\ldots ,dj\,$, the matrix $%
J^{n}$ can be written as the semi-infinite block matrix%
\begin{equation*}
J^{n}=%
\begin{bmatrix}
J_{11}^{n} & J_{12}^{n} & \cdots \\ 
J_{21}^{n} & J_{22}^{n} & \cdots \\ 
\vdots & \vdots & \ddots%
\end{bmatrix}%
.
\end{equation*}%
We can now formulate our first important result in this Section. For more
details we refer the reader to \cite[Sec. 1.2]{BBF-JMAA-10} and \cite%
{AKI-CA-00}.

Let $\{\mathcal{S}_{n}\}$\ be a symmetric type~II vector multiple polynomial
sequence orthogonal with respect to the $d$--symmetric vector of functionals 
$\mathcal{V}$. Following \cite[Th. 7]{BCF-NA-10}, the matrix generating
function associated to $\mathcal{V}$, and the Weyl function associated with $%
J$, the block Jacobi matrix corresponding to $\{\mathcal{S}_{n}\}$, can be\
put in correspondence by means of the matrix expression%
\begin{equation}
\mathcal{F}(z)=R_{J}(z)\mathcal{V}(\mathcal{X}_{0})\,,  \label{[Sec3]-FyRV00}
\end{equation}%
where, for the $d$--symmetric case, $\mathcal{V}(\mathcal{X}_{0})=\mathcal{V}%
(\mathcal{S}_{0})=\mathcal{V}_{0}^{0}$, which is explicitly given in (\ref%
{[Sec2]-V00}).

First we study the case $d=2$, and next we consider more general situations
for $d>2$ functional entries in $\mathcal{V}$. From Lemma \ref%
{[SEC2]-TH2-MOM-v1v2}, we obtain the entries for the representation of the
Stieltjes matrix function $\mathcal{F}(z)$, associated with~$\mathcal{V}=%
\begin{bmatrix}
v^{1} & v^{2}%
\end{bmatrix}%
$, as%
\begin{multline}
\mathcal{F}(z)=\sum_{n=0}^{\infty }{%
\begin{bmatrix}
v^{1}(x^{3n}) & \frac{v_{2,0}\cdot v^{1}(x^{3n})}{v_{1,0}} \\ 
0 & {v^{2}(x^{3n+1})}%
\end{bmatrix}%
}/{z^{3n+1}}+\sum_{n=0}^{\infty }{%
\begin{bmatrix}
{0} & {v^{2}(x^{3n+1})} \\ 
0 & {0}%
\end{bmatrix}%
}/{z^{3n+2}}  \label{[Sec2]-DescFen3} \\
+\sum_{n=0}^{\infty }{%
\begin{bmatrix}
{0} & {0} \\ 
v^{1}(x^{3n+3}) & \frac{v_{2,0}\cdot v^{1}(x^{3n+3})}{v_{1,0}}%
\end{bmatrix}%
}/{z^{3n+3}}\,.
\end{multline}%
Notice that we have $\mathcal{F}(z)=\mathcal{F}_{1}(z)+\mathcal{F}_{2}(z)+%
\mathcal{F}_{3}(z)$. The following theorem shows that we can obtain $%
\mathcal{F}(z)$ in our particular case, analyzing two different situations.

\begin{teo}
\label{[SEC2]-TEO-Fv2}Let $\mathcal{V}=%
\begin{bmatrix}
v^{1} & v^{2}%
\end{bmatrix}%
^{T}$ be a symmetric vector of functionals, with $d=2$. Then the Weyl
function is given by 
\begin{equation*}
\mathcal{R}_{J}(z)=\sum_{n=0}^{\infty }\frac{%
\begin{bmatrix}
\frac{v^{1}(x^{3n})}{v_{1,0}} & 0 \\ 
0 & \frac{v^{2}(x^{3n+1})}{v_{2,1}}%
\end{bmatrix}%
}{z^{3n+1}}+\sum_{n=0}^{\infty }\frac{%
\begin{bmatrix}
{0} & \frac{v^{2}(x^{3n+1})}{v_{2,1}} \\ 
0 & {0}%
\end{bmatrix}%
}{z^{3n+2}}+\sum_{n=0}^{\infty }\frac{%
\begin{bmatrix}
{0} & {0} \\ 
\frac{v^{1}(x^{3n+3})}{v_{1,0}} & 0%
\end{bmatrix}%
}{z^{3n+3}}
\end{equation*}
\end{teo}

\begin{proof}
It is enough to multiply (\ref{[Sec2]-DescFen3}) by $(\mathcal{V}%
_{0}^{0})^{-1}$. The explicit expression for $\mathcal{V}_{0}^{0}$ is given
in (\ref{[Sec2]-V00}).
\end{proof}

Computations considering $d>2$ functionals, can be cumbersome but doable.
The matrix generating function $\mathcal{F}$ (as well as the Weyl function),
will be the sum of $d+1$ matrix terms, i.e. \ $\mathcal{F}(z)=\mathcal{F}%
_{1}(z)+\cdots +\mathcal{F}_{d+1}(z)$, each of them of size $d\times d$.

Let us now outline the very interesting structure of $\mathcal{R}_{J}(z)$
for the general case of $d$ functionals. We shall describe the structure of $%
\mathcal{R}_{J}(z)$ for $d=3$, comparing the situation with the general
case. Let $\ast $ denote every non-zero entry in a given matrix. Thus, there
will be four $J_{11}$\ matrices of size $3\times 3$. Here, and in the
general case, the first matrix $[J_{11}^{(d+1)n}]_{d\times d}$ will allways
be diagonal, as follows%
\begin{equation*}
\lbrack J_{11}^{4n}]_{3\times 3}=%
\begin{bmatrix}
\ast & 0 & 0 \\ 
0 & \ast & 0 \\ 
0 & 0 & \ast%
\end{bmatrix}%
.
\end{equation*}%
Indeed, observe that for $n=0$, $[J_{11}^{(d+1)n}]_{d\times d}=I_{d}$. Next,
we have%
\begin{equation*}
\lbrack J_{11}^{4n+1}]_{3\times 3}=%
\begin{bmatrix}
0 & \ast & 0 \\ 
0 & 0 & \ast \\ 
0 & 0 & 0%
\end{bmatrix}%
.
\end{equation*}%
From (\ref{[Sec2]-HTRR-Symm}) we know that the \textquotedblleft
distance\textquotedblright\ between the two extreme non-zero diagonals of $J$
will allways consist of $d$ zero diagonals. It directly implies that, also
in the general case, every entry in the last row of $[J_{11}^{(d+1)n}]_{d%
\times d}$ will always be zero, and therefore the unique non-zero entries in 
$[J_{11}^{(d+1)n}]_{d\times d}$ will be one step over the main diagonal.

Next we have%
\begin{equation*}
\lbrack J_{11}^{4n+2}]_{3\times 3}=%
\begin{bmatrix}
0 & 0 & \ast \\ 
0 & 0 & 0 \\ 
\ast & 0 & 0%
\end{bmatrix}%
.
\end{equation*}%
Notice that for every step, the main diagonal in $[J_{11}^{(d+1)n}]_{d\times
d}$ goes \textquotedblleft one diagonal\textquotedblright\ up, but the other
extreme diagonal of $J$ is also moving upwards, with exactly $d$ zero
diagonals between them. It directly implies that, no matter the number of
functionals, only the lowest-left element of $[J_{11}^{(d+1)n+2}]_{d\times
d} $ will be different from zero. Finally, we have%
\begin{equation*}
\lbrack J_{11}^{4n+3}]_{3\times 3}=%
\begin{bmatrix}
0 & 0 & 0 \\ 
\ast & 0 & 0 \\ 
0 & \ast & 0%
\end{bmatrix}%
.
\end{equation*}%
Here, the last non-zero entry of the main diagonal in $[J_{11}^{4n}]_{3%
\times 3}$ vanishes. In the general case, it will occur exactly at step $%
[J_{11}^{(d+1)n+(d-1)}]_{d\times d}$, in which only the uppest-right entry
is different from zero. Meanwhile, in matrices $[J_{11}^{(d+1)n+3}]_{d\times
d}$ up to $[J_{11}^{(d+1)n+(d-2)}]_{d\times d}$ will be non-zero entries of
the two extreme diagonals. In this last situation, the non-zero entries of $%
[J_{11}^{(d+1)n+(d-1)}]_{d\times d}$, will always be exactly one step under
the main diagonal.


\section{The symmetrization problem for multiple OP}

\label{[Section-4]-SymProb}


Throughout this section, let $\{A_{n}^{1}\}$\ be an arbitrary and not
necesarily symmetric sequence of type~II multiple orthogonal polynomials,
satisfying a $(d+2)$--term recurrence relation with known recurrence
coefficients, and let $J_{1}$ be the corresponding $(d+2)$--banded matrix.
Let $J_{1}$ be such that the following $LU$ factorization%
\begin{equation}
J_{1}=LU=L_{1}L_{2}\cdots L_{d}U  \label{[Sec4]-xdFactorization}
\end{equation}%
is unique, where $U$ is an upper two--banded, semi--infinite, and invertible
matrix, $L$ is a $(d+1)$--lower triangular semi--infinite with ones in the
main diagonal, and every $L_{i}$, $i=1,\ldots ,d$ is a lower two--banded,
semi--infinite, and invertible matrix with ones in the main diagonal. We
follow \cite[Def. 3]{BBF-JMAA-13}, where the authors generalize the concept
of Darboux transformation to general Hessenberg banded matrices, in assuming
that any circular permutation of $L_{1}L_{2}\cdots L_{d}U$\ is a Darboux
transformation of $J_{1}$. Thus we have $d$ possible Darboux transformations
of $J_{1}$, say $J_{j}$, $j=2,\ldots ,d+1$, with $J_{2}=L_{2}\cdots
L_{d}UL_{1}\,$, $J_{3}=L_{3}\cdots L_{d}UL_{1}L_{2}\,$,\ldots , $%
J_{d+1}=UL_{1}L_{2}\cdots L_{d}\,$.

Next, we solve the so called \textit{symmetrization problem} in the theory
of multiple orthogonal polynomials, i.e., starting with $\{A_{n}^{1}\}$, we
find a total $d+1$ type~II multiple orthogonal polynomial sequences $%
\{A_{n}^{j}\}$, $j=1,\ldots ,d+1$, satisfying $(d+2)$--term recurrence
relation with known recurrence coefficients, which can be used to construct
a new $d$--symmetric type~II multiple orthogonal polynomial sequence $%
\{S_{n}\}$. It is worth pointing out that all the aforesaid sequences $%
\{A_{n}^{j}\}$, $j=1,\ldots ,d+1$\ are of the same kind, with the same
number of elements in their respectives $(d+2)$--term recurrence relations,
and multiple orthogonal with respect to the same number of functionals $d$.

\begin{teo}
\label{Th-Symm-Directo}Let $\{A_{n}^{1}\}$\ be an arbitrary and not
necesarily symmetric sequence of type~II multiple orthogonal polynomials as
stated above. Let $J_{j}$, $j=2,\ldots ,d+1$, be the Darboux transformations
of $J_{1}$ given by the $d$ cyclid permutations of the matrices in the right
hand side of (\ref{[Sec4]-xdFactorization}). Let $\{A_{n}^{j}\}$, $%
j=2,\ldots ,d+1$, $d$ new families of type~II multiple orthogonal
polynomials satisfying $(d+2)$--term recurrence relations given by the
matrices $J_{j}$, $j=2,\ldots ,d+1$. Then, the sequence $\{S_{n}\}$\ defined
by%
\begin{equation}
\left\{ 
\begin{array}{l}
S_{(d+1)n}(x)=A_{n}^{1}(x^{d+1}), \\ 
S_{(d+1)n+1}(x)=xA_{n}^{2}(x^{d+1}), \\ 
\cdots \\ 
S_{(d+1)n+d}(x)=x^{d}A_{n}^{d+1}(x^{d+1}),%
\end{array}%
\right.  \label{[Sec4]-xdDecomp}
\end{equation}%
is a $d$--symmetric sequence of type~II multiple orthogonal polynomials.
\end{teo}

\begin{proof}
Let $\{A_{n}^{1}\}$ satify the $(d+2)$--term recurrence relation given by (%
\ref{[Sec2]-HTRR-MultOP})%
\begin{equation*}
xA_{n+d}^{1}(x)=A_{n+d+1}^{1}(x)+b_{n+d}^{[1]}A_{n+d}^{1}(x)+\sum_{\nu
=0}^{d-1}c_{n+d-\nu }^{d-1-\nu ,[1]}A_{n+d-1-\nu }^{1}(x)\,,
\end{equation*}%
$c_{n+1}^{0,[1]}\neq 0$ for $n\geq 0$, with the initial conditions $%
A_{0}^{1}(x)=1$, $A_{1}^{1}(x)=x-b_{0}^{[1]}$,~and%
\begin{equation*}
A_{n}^{1}(x)=(x-b_{n-1}^{[1]})A_{n-1}^{1}(x)-\sum_{\nu =0}^{n-2}c_{n-1-\nu
}^{d-1-\nu ,[1]}A_{n-2-\nu }^{1}(x)\,,\ \ 2\leq n\leq d\,,
\end{equation*}%
with known recurrence coefficients. Hence, in a matrix notation, we have%
\begin{equation}
x\mathbf{\underline{A}}^{1}=J_{1}\,\mathbf{\underline{A}}^{1},
\label{[Sec4]-J1A1}
\end{equation}%
where%
\begin{equation*}
J_{1}=%
\begin{bmatrix}
b_{d}^{[1]} & 1 &  &  &  &  &  \\ 
c_{d+1}^{d-1,[1]} & b_{d+1}^{[1]} & 1 &  &  &  &  \\ 
c_{d+1}^{d-2,[1]} & c_{d+2}^{d-1,[1]} & b_{d+2}^{[1]} & 1 &  &  &  \\ 
\vdots  & c_{d+2}^{d-2,[1]} & c_{d+3}^{d-1,[1]} & b_{d+3}^{[1]} & 1 &  &  \\ 
c_{d+1}^{0,[1]} & \cdots  & c_{d+3}^{d-2,[1]} & c_{d+4}^{d-1,[1]} & 
b_{d+4}^{[1]} & 1 &  \\ 
& \ddots  & \ddots  & \ddots  & \ddots  & \ddots  & \ddots 
\end{bmatrix}%
.
\end{equation*}%
Following \cite{BBF-JMAA-13} (see also \cite{Keller-94}), the unique $LU$
factorization for the square $(d+2)$--banded semi-infinite Hessenberg matrix 
$J_{1}$ is such that%
\begin{equation*}
U=%
\begin{bmatrix}
\gamma _{1} & 1 &  &  &  \\ 
& \gamma _{d+2} & 1 &  &  \\ 
&  & \gamma _{2d+3} & 1 &  \\ 
&  &  & \gamma _{3d+4} & \ddots  \\ 
&  &  &  & \ddots 
\end{bmatrix}%
\end{equation*}%
is an upper two-banded, semi-infinite, and invertible matrix, and $L$ is a $%
(d+1)$--lower triangular, semi--infinite, and invertible matrix with ones in
the main diagonal. It is clear that the entries in $L$ and $U$ depend
entirely on the known entries of $J_{1}$. Thus, we rewrite (\ref{[Sec4]-J1A1}%
) as%
\begin{equation}
x\mathbf{\underline{A}}^{1}=LU\,\mathbf{\underline{A}}^{1}.
\label{[Sec4]-LUA1}
\end{equation}%
Next, we define a\ new sequence of polynomials $\{A_{n}^{d+1}\}$ by $x%
\mathbf{\underline{A}}^{d+1}=U\mathbf{\underline{A}}^{1}$. Multiplying both
sides of (\ref{[Sec4]-LUA1}) by the matrix $U$, we have%
\begin{equation*}
x\left( U\,\mathbf{\underline{A}}^{1}\right) =UL\left( U\,\mathbf{\underline{%
A}}^{1}\right) =x(x\mathbf{\underline{A}}^{d+1})=UL(x\mathbf{\underline{A}}%
^{d+1}),
\end{equation*}%
and pulling out $x$ we get%
\begin{equation}
x\mathbf{\underline{A}}^{d+1}=UL\,\mathbf{\underline{A}}^{d+1},
\label{[Sec4]-UAd1}
\end{equation}%
which is the matrix form of the $(d+2)$--term recurrence relation satisfied
by the new type~II multiple polynomial sequence $\{A_{n}^{d+1}\}$.

Since $L$ is given by (see \cite{BBF-JMAA-13})%
\begin{equation*}
L=L_{1}L_{2}\cdots L_{d},
\end{equation*}%
where every $L_{j}$ is the lower two-banded, semi-infinite, and invertible
matrix%
\begin{equation*}
L_{j}=%
\begin{bmatrix}
1 &  &  &  &  \\ 
\gamma _{d-j} & 1 &  &  &  \\ 
& \gamma _{2d+1-j} & 1 &  &  \\ 
&  & \gamma _{3d+2-j} & 1 &  \\ 
&  &  & \ddots & \ddots%
\end{bmatrix}%
\end{equation*}%
with ones in the main diagonal, it is also clear that the entries in $L_{j}$%
\ will depend on the known entries in $J_{1}$. Under the same hypotheses, we
can define new $d-1$ polinomial sequences starting with $\mathbf{\underline{A%
}}^{1}$ as follows: $\mathbf{\underline{A}}^{2}=L_{1}^{-1}\mathbf{\underline{%
A}}^{1}$, $\mathbf{\underline{A}}^{3}=L_{2}^{-1}\mathbf{\underline{A}}^{2}$,
\ldots , $\mathbf{\underline{A}}^{d}=L_{d-1}^{-1}\mathbf{\underline{A}}%
^{d-1} $ up to $\mathbf{\underline{A}}^{d+1}=L_{d}^{-1}\mathbf{\underline{A}}%
^{d}$. That is, $\mathbf{\underline{A}}^{j+1}=L_{j}^{-1}\mathbf{\underline{A}%
}^{j}$, $j=1,\ldots d-1$. Combining this fact with (\ref{[Sec4]-UAd1}) we
deduce%
\begin{eqnarray*}
x\mathbf{\underline{A}}^{d} &=&L_{d}UL_{1}L_{2}\cdots L_{d-1}\mathbf{%
\underline{A}}^{d}, \\
x\mathbf{\underline{A}}^{d-1} &=&L_{d-1}L_{d}UL_{1}L_{2}\cdots L_{d-2}%
\mathbf{\underline{A}}^{d-1},
\end{eqnarray*}%
up to the known expression (\ref{[Sec4]-J1A1})%
\begin{equation*}
x\mathbf{\underline{A}}^{1}=L_{1}L_{2}\cdots L_{d}U\mathbf{\underline{A}}%
^{1}.
\end{equation*}%
The above expresions mean that all these $d+1$ sequences $\{A^{j}\}$, $%
j=1,\ldots ,d+1$ are in turn type~II multiple orthogonal polynomials.
Finally, from these $d+1$ sequences, we construct the type~II polynomials in
the sequence $\{S_{n}\}$ as (\ref{[Sec4]-xdDecomp}). Note that, according to
(\ref{[Sec2]-d-symmetric-Sn}), it directly follows that $\{S_{n}\}$ is a $d$%
--symmetric type~II multiple orthogonal polynomial sequence, which proves
our assertion.
\end{proof}

Next, we state the converse of the above theorem. That is, given a sequence
of type~II $d$--symmetric multiple orthogonal polynomials $\{S_{n}\}$
satisfying the high--term recurrence relation~(\ref{[Sec2]-HTRR-Symm}), we
find a set of $d+1$ polynomial families $\{A_{n}^{j}\}$, $j=1,\ldots ,d+1$
of not necessarily symmetric type~II multiple orthogonal polynomials,
satisfying in turn $(d+2)$--term recurrence relations, so o they are
themselves sequences of type~II multiple orthogonal polynomials. When $d=2$,
this construction goes back to the work of Douak and Maroni (see \cite%
{DM-A-92}).

\begin{teo}
\label{Th-Symm-Inverso} Let $\{S_{n}\}$ be a $d$--symmetric sequence of
type~II multiple orthogonal polynomials satisfiyng the corresponding
high--order recurrence relation (\ref{[Sec2]-HTRR-Symm}), and $\{A_{n}^{j}\}$%
, $j=1,\ldots ,d+1$, the sequences of polynomials given by (\ref%
{[Sec4]-xdDecomp}). Then, each sequence $\{A_{n}^{j}\}$, $j=1,\ldots ,d+1$,
satisfies the $(d+2)$--term recurrence relation%
\begin{equation*}
xA_{n+d}^{j}(x)=A_{n+d+1}^{j}(x)+b_{n+d}^{[j]}A_{n+d}^{j}(x)+\sum_{\nu
=0}^{d-1}c_{n+d-\nu }^{d-1-\nu ,[j]}A_{n+d-1-\nu }^{j}(x)\,,
\end{equation*}%
$c_{n+1}^{0,[j]}\neq 0$ for $n\geq 0$, with initial conditions $%
A_{0}^{j}(x)=1$, $A_{1}^{j}(x)=x-b_{0}^{[j]}$,%
\begin{equation*}
A_{n}^{j}(x)=(x-b_{n-1}^{[j]})A_{n-1}^{j}(x)-\sum_{\nu =0}^{n-2}c_{n-1-\nu
}^{d-1-\nu ,[j]}A_{n-2-\nu }^{j}(x)\,,\ \ 2\leq n\leq d\,,
\end{equation*}%
and therefore they are type~II multiple orthogonal polynomial sequences.
\end{teo}

\begin{proof}
Since $\{S_{n}\}$ is a $d$--symmetric multiple orthogonal sequence, it
satisfies (\ref{[Sec2]-HTRR-Symm}) with $S_{n}(x)=x^{n}$ for $0\leq n\leq d$%
. Shifting, for convenience, the multi--index in (\ref{[Sec2]-HTRR-Symm}) as 
$n\rightarrow (d+1)n-d+j,\quad j=0,1,2,....d$, we obtain the equivalent
system of $(d+1)$ equations%
\begin{equation*}
\left\{ 
\begin{array}{cc}
xS_{(d+1)n}(x)=S_{(d+1)n+1}(x)+\gamma _{(d+1)n-d+1}S_{(d+1)n-d}(x)\,, & j=0,
\\ 
xS_{(d+1)n+1}(x)=S_{(d+1)n+2}(x)+\gamma _{(d+1)n-d+2}S_{(d+1)n-d+1}(x)\,, & 
j=1, \\ 
\vdots & \vdots \\ 
xS_{(d+1)n+d-1}(x)=S_{(d+1)n+d}(x)+\gamma _{(d+1)n}S_{(d+1)n-1}(x)\,, & 
j=d-1, \\ 
xS_{(d+1)n+d}(x)=S_{(d+1)n+(d+1)}(x)+\gamma _{(d+1)n+1}S_{(d+1)n}(x)\,, & 
j=d.%
\end{array}%
\right. \,.
\end{equation*}%
Substituting (\ref{[Sec4]-xdDecomp}) into the above expressions, and
replacing $x^{d+1}\rightarrow x$, we get the following system of $(d+1)$
equations%
\begin{equation}
\left\{ 
\begin{array}{cc}
1) & A_{n}^{1}(x)=A_{n}^{2}(x)+\gamma _{(d+1)n-(d-1)}A_{n-1}^{2}(x)\,, \\ 
2) & A_{n}^{2}(x)=A_{n}^{3}(x)+\gamma _{(d+1)n-(d-2)}A_{n-1}^{3}(x)\,, \\ 
\vdots & \vdots \\ 
d) & A_{n}^{d}(x)=A_{n}^{d+1}(x)+\gamma _{(d+1)n}A_{n-1}^{d+1}(x), \\ 
d+1) & xA_{n}^{d+1}(x)=A_{n+1}^{1}(x)+\gamma _{(d+1)n+1}A_{n}^{1}(x)\,.%
\end{array}%
\right.  \label{[Sec4]-9-Relaciones}
\end{equation}%
Notice that having $x=0$, we define the $\gamma $ as%
\begin{equation*}
\gamma _{(d+1)n+1}=\frac{-A_{n+1}^{1}(0)}{A_{n}^{1}(0)}.
\end{equation*}%
In the remainder of this section we deal with the matrix representation of
each equation in (\ref{[Sec4]-9-Relaciones}). Notice that the first $d$
equations of (\ref{[Sec4]-9-Relaciones}), namely $A_{n}^{j}=A_{n}^{j+1}+%
\gamma _{(d+1)n+(j-d)}A_{n-1}^{j+1}$, \ $j=1,\ldots ,d$, can be written in
the matrix way%
\begin{equation}
\mathbf{\underline{A}}^{j}=L_{j}\,\mathbf{\underline{A}}^{j+1},
\label{[Sec4]-LowerLj}
\end{equation}%
where $L_{j}$ is the lower two-banded, semi-infinite, and invertible matrix%
\begin{equation*}
L_{j}=%
\begin{bmatrix}
1 &  &  &  &  \\ 
\gamma _{d-j} & 1 &  &  &  \\ 
& \gamma _{2d+1-j} & 1 &  &  \\ 
&  & \gamma _{3d+2-j} & 1 &  \\ 
&  &  & \ddots & \ddots%
\end{bmatrix}%
,
\end{equation*}%
and%
\begin{equation*}
\mathbf{\underline{A}}^{j}=%
\begin{bmatrix}
A_{0}^{j}(x) & A_{1}^{j}(x) & A_{2}^{j}(x) & \cdots%
\end{bmatrix}%
^{T}.
\end{equation*}%
Similarly, the $d+1)$--th equation in (\ref{[Sec4]-9-Relaciones}) can be
expressed as%
\begin{equation}
x\mathbf{\underline{A}}^{d+1}=U\,\mathbf{\underline{A}}^{1},
\label{[Sec4]-UpperU}
\end{equation}%
where $U$ is the upper two-banded, semi-infinite, and invertible matrix%
\begin{equation*}
U=%
\begin{bmatrix}
\gamma _{1} & 1 &  &  &  \\ 
& \gamma _{d+2} & 1 &  &  \\ 
&  & \gamma _{2d+3} & 1 &  \\ 
&  &  & \gamma _{3d+4} & \ddots \\ 
&  &  &  & \ddots%
\end{bmatrix}%
.
\end{equation*}%
It is clear that the entries in the above matrices $L_{j}$ and $U$ are given
in terms of the recurrence coefficients $\gamma _{n+1}$ for $\{S_{n}\}$
given (\ref{[Sec2]-HTRR-Symm}). From (\ref{[Sec4]-LowerLj}) we have $\mathbf{%
\underline{A}}^{1}$ defined in terms of $\mathbf{\underline{A}}^{2}$, as $%
\mathbf{\underline{A}}^{1}=L_{1}\mathbf{\underline{A}}^{2}$. Likewise $%
\mathbf{\underline{A}}^{2}$ in terms of $\mathbf{\underline{A}}^{3}$ as $%
\mathbf{\underline{A}}^{2}=L_{2}\mathbf{\underline{A}}^{3}$, and so on up to 
$j=d$. Thus, it is easy to see that,%
\begin{equation*}
\mathbf{\underline{A}}^{1}=L_{1}L_{2}\cdots L_{d}\mathbf{\underline{A}}%
^{d+1}.
\end{equation*}%
Next, we multiply by $x$ both sides of the above expression, and we apply (%
\ref{[Sec4]-UpperU}) to obtain%
\begin{equation*}
x\mathbf{\underline{A}}^{1}=L_{1}L_{2}\cdots L_{d}U\,\mathbf{\underline{A}}%
^{1}.
\end{equation*}%
Since each $L_{j}$ and $U$ are lower and upper two-banded semi-infinite
matrices, it follows easily that $L_{1}L_{2}\cdots L_{d}$ is a lower
triangular $(d+1)$--banded matrix with ones in the main diagonal, so the
above decomposition is indeed a $LU$ decomposition of certain $(d+2)$%
--banded Hessenberg matrix $J_{1}=L_{1}L_{2}\cdots L_{d}U$\ (see for
instance \cite{BBF-JMAA-13} and \cite[Sec. 3.2 and 3.3]{Keller-94}). The
values of the entries of $J_{1}$ come from the usual definition for matrix
multiplication, matching every entry in $J_{1}=L_{1}L_{2}\cdots L_{d}U$,
with $L_{j}\,U$ given in (\ref{[Sec4]-LowerLj}) and (\ref{[Sec4]-UpperU})
respectively.

On the other hand, starting with $\mathbf{\underline{A}}^{2}=L_{2}\mathbf{%
\underline{A}}^{3}$ instead of $\mathbf{\underline{A}}^{1}$, and proceeding
in the same fashion as above, we can reach%
\begin{eqnarray*}
x\mathbf{\underline{A}}^{2} &=&L_{2}\cdots L_{d}UL_{1}\,\mathbf{\underline{A}%
}^{2}, \\
x\mathbf{\underline{A}}^{3} &=&L_{3}\cdots L_{d}UL_{1}L_{2}\,\mathbf{%
\underline{A}}^{3}, \\
&&\vdots \\
x\mathbf{\underline{A}}^{d+1} &=&UL_{1}L_{2}\cdots L_{d}\,\mathbf{\underline{%
A}}^{d+1}.
\end{eqnarray*}%
Observe that $J_{j}$ denotes a particular circular permutation of the matrix
product $L_{1}L_{2}\cdots L_{d}U$ . Thus, we have $J_{1}=L_{1}L_{2}\cdots
L_{d}U\,$, $J_{2}=L_{2}\cdots L_{d}UL_{1}\,$, \ldots , $J_{d+1}=UL_{1}L_{2}%
\cdots L_{d}\,$. Using this notation, $J_{j}$ is the matrix representation
of the operator of multiplication by $x$ in%
\begin{equation}
x\mathbf{\underline{A}}^{j}=J_{j}\mathbf{\underline{A}}^{j},
\label{[Sec4]-MatrixJj}
\end{equation}%
which from (\ref{[Sec2]-HTRR-MultOP}) directly implies that each polynomial
sequence $\{A^{j}\}$, $j=1,\ldots d+1$, satisfies a $(d+2)$--term recurrence
relation as in the statement of the theorem, with coefficients given in
terms of the recurrence coefficients $\gamma _{n+1}\,$, from the high--term
recurrence relation (\ref{[Sec2]-HTRR-Symm}) satisfied by the symmetric
sequence $\{S_{n}\}$.

This completes the proof.
\end{proof}


\section{Matrix multiple orthogonality}

\label{[Section-5]-FavardTh}


For an arbitrary system of type~II vector multiple polynomials $\{P_{n}\}$
orthogonal with respect to certain vector of functionals $\mathcal{U}=%
\begin{bmatrix}
u^{1} & u^{2}%
\end{bmatrix}%
^{T}$, with%
\begin{equation}
\mathcal{P}_{n}=%
\begin{bmatrix}
P_{3n}(x) & P_{3n+1}(x) & P_{3n+2}(x)%
\end{bmatrix}%
^{T},  \label{[Sec5]-P3n}
\end{equation}%
there exists a matrix decomposition%
\begin{equation}
\mathcal{P}_{n}=W_{n}(x^{3})\mathcal{X}_{n}\rightarrow 
\begin{bmatrix}
P_{3n}(x) \\ 
P_{3n+1}(x) \\ 
P_{3n+2}(x)%
\end{bmatrix}%
=W_{n}(x^{3})%
\begin{bmatrix}
1 \\ 
x \\ 
x^{2}%
\end{bmatrix}%
,  \label{[Sec5]-PnWnX}
\end{equation}%
with $W_{n}$ being the matrix polynomial (see \cite{MM-MJM13})%
\begin{equation}
W_{n}(x)=%
\begin{bmatrix}
A_{n}^{1}(x) & A_{n-1}^{2}(x) & A_{n-1}^{3}(x) \\ 
B_{n}^{1}(x) & B_{n}^{2}(x) & B_{n-1}^{3}(x) \\ 
C_{n}^{1}(x) & C_{n}^{2}(x) & C_{n}^{3}(x)%
\end{bmatrix}%
.  \label{[Sec5]-MatrixWn}
\end{equation}%
Throughout this Section, for simplicity of computations, we assume $d=2$ for
the vector of functionals $\mathcal{U}$, but the same results can be easily
extended for an arbitrary number of functionals. We first show that, if a
sequence of type~II multiple $2$--orthogonal polynomials $\{P_{n}\}$ satisfy
a recurrence relation like (\ref{[Sec2]-4TRR-P}), then there exists a
sequence of matrix polynomials $\{W_{n}\}$, $W_{n}(x)\in \mathbb{P}^{3\times
3}$ associated to $\{P_{n}\}$ by (\ref{[Sec5]-P3n}) and (\ref{[Sec5]-PnWnX}%
), satisfying a matrix four term recurrence relation with matrix
coefficients.

\begin{teo}
\label{[SEC5]-TH1}Let $\{P_{n}\}$ be a sequence of type~II multiple
polynomials, $2$--orthogonal with respect to the system of functionals $%
\{u^{1},u^{2}\}$, i.e., satisfying the four--term type recurrence relation (%
\ref{[Sec2]-4TRR-P}). Let $\{W_{n}\}$, $W_{n}(x)\in \mathbb{P}^{3\times 3}$
associated to $\{P_{n}\}$ by (\ref{[Sec5]-P3n}) and (\ref{[Sec5]-PnWnX}).
Then, the matrix polynomials $W_{n}$ satisfy a matrix four term recurrence
relation with matrix coefficients.
\end{teo}

\begin{proof}
We first prove that the sequence of vector polynomials $\{W_{n}\}$ satisfy a
four term recurrence relation with matrix coefficients, starting from the
fact that $\{P_{n}\}$ satisfy (\ref{[Sec2]-4TRR-P}). In order to get this
result, we use the matrix interpretation of multiple orthogonality described
in Section \ref{[Section-2]-Defs}. We know that the sequence of type~II
multiple $2$--orthogonal polynomials $\{P_{n}\}$ satisfy the four term
recurrence relation (\ref{[Sec2]-4TRR-P}). From (\ref{[Sec5]-P3n}), using
the matrix interpretation for multiple orthogonality, the above expression
can be seen as the matrix three term recurrence relation%
\begin{equation*}
x\mathcal{P}_{n}=A\mathcal{P}_{n+1}+B_{n}\mathcal{P}_{n}+C_{n}\mathcal{P}%
_{n-1}
\end{equation*}%
or, equivalently%
\begin{equation*}
x%
\begin{bmatrix}
P_{3n} \\ 
P_{3n+1} \\ 
P_{3n+2}%
\end{bmatrix}%
=%
\begin{bmatrix}
0 & 0 & 0 \\ 
0 & 0 & 0 \\ 
1 & 0 & 0%
\end{bmatrix}%
\begin{bmatrix}
P_{3n+3} \\ 
P_{3n+4} \\ 
P_{3n+5}%
\end{bmatrix}%
\end{equation*}%
\begin{equation*}
+%
\begin{bmatrix}
b_{3n} & 1 & 0 \\ 
c_{3n+1} & b_{3n+1} & 1 \\ 
d_{3n+2} & c_{3n+2} & b_{3n+2}%
\end{bmatrix}%
\begin{bmatrix}
P_{3n} \\ 
P_{3n+1} \\ 
P_{3n+2}%
\end{bmatrix}%
+%
\begin{bmatrix}
0 & d_{3n} & c_{3n} \\ 
0 & 0 & d_{3n+1} \\ 
0 & 0 & 0%
\end{bmatrix}%
\begin{bmatrix}
P_{3n-3} \\ 
P_{3n-2} \\ 
P_{3n-1}%
\end{bmatrix}%
\end{equation*}%
Multiplying the above expression by $x$ we get%
\begin{eqnarray}
x^{2}\mathcal{P}_{n} &=&Ax\mathcal{P}_{n+1}+B_{n}x\mathcal{P}_{n}+C_{n}x%
\mathcal{P}_{n-1}  \notag \\
&=&A\left[ x\mathcal{P}_{n+1}\right] +B_{n}\left[ x\mathcal{P}_{n}\right]
+C_{n}\left[ x\mathcal{P}_{n-1}\right]  \notag \\
&=&AA\mathcal{P}_{n+2}+\left[ AB_{n+1}+B_{n}A\right] \mathcal{P}_{n+1}
\label{[Sec5]-PrimABC} \\
&&+\left[ AC_{n+1}+B_{n}B_{n}+C_{n}A\right] \mathcal{P}_{n}  \notag \\
&&+\left[ B_{n}C_{n}+C_{n}B_{n-1}\right] \mathcal{P}_{n-1}+C_{n}C_{n-1}%
\mathcal{P}_{n-2}  \notag
\end{eqnarray}%
The matrix $A$ is nilpotent, so $AA$ is the zero matrix of size $3\times 3$.
Having%
\begin{equation*}
\begin{array}{ll}
A_{n}^{\langle 1\rangle }=AB_{n+1}+B_{n}A, & B_{n}^{\langle 1\rangle
}=AC_{n+1}+B_{n}B_{n}+C_{n}A, \\ 
C_{n}^{\langle 1\rangle }=B_{n}C_{n}+C_{n}B_{n-1}, & D_{n}^{\langle 1\rangle
}=C_{n}C_{n-1},%
\end{array}%
\end{equation*}%
where the entries of $A_{n}^{\langle 1\rangle }$, $B_{n}^{\langle 1\rangle }$%
, and $C_{n}^{\langle 1\rangle }$ can be easily obtained using a
computational software as Mathematica$^{\circledR }$ or Maple$^{\circledR }$%
, from the entries of $A_{n}$, $B_{n}$, and $C_{n}$. Thus, we can rewrite (%
\ref{[Sec5]-PrimABC}) as%
\begin{equation*}
x^{2}\mathcal{P}_{n}=A_{n}^{\langle 1\rangle }\mathcal{P}_{n+1}+B_{n}^{%
\langle 1\rangle }\mathcal{P}_{n}+C_{n}^{\langle 1\rangle }\mathcal{P}%
_{n-1}+D_{n}^{\langle 1\rangle }\mathcal{P}_{n-2}.
\end{equation*}%
We now continue in this fashion,multiplying again by $x$%
\begin{eqnarray*}
x^{3}\mathcal{P}_{n} &=&A_{n}^{\langle 1\rangle }x\mathcal{P}%
_{n+1}+B_{n}^{\langle 1\rangle }x\mathcal{P}_{n}+C_{n}^{\langle 1\rangle }x%
\mathcal{P}_{n-1}+D_{n}^{\langle 1\rangle }x\mathcal{P}_{n-2} \\
&=&A_{n}^{\langle 1\rangle }A\mathcal{P}_{n+2}+\left[ A_{n}^{\langle
1\rangle }B_{n+1}+B_{n}^{\langle 1\rangle }A\right] \mathcal{P}_{n+1} \\
&&+\left[ A_{n}^{\langle 1\rangle }C_{n+1}+B_{n}^{\langle 1\rangle
}B_{n}+C_{n}^{\langle 1\rangle }A\right] \mathcal{P}_{n} \\
&&+\left[ B_{n}^{\langle 1\rangle }C_{n}+C_{n}^{\langle 1\rangle
}B_{n-1}+D_{n}^{\langle 1\rangle }A\right] \mathcal{P}_{n-1} \\
&&+\left[ C_{n}^{\langle 1\rangle }C_{n-1}+D_{n}^{\langle 1\rangle }B_{n-2}%
\right] \mathcal{P}_{n-2}+D_{n}^{\langle 1\rangle }C_{n-2}\mathcal{P}_{n-3}
\end{eqnarray*}%
The matrix products $A_{n}^{\langle 1\rangle }A$ and $D_{n}^{\langle
1\rangle }C_{n-2}$ both give the zero matrix of size $3\times 3$, and the
remaining matrix coefficients are%
\begin{equation}
\begin{array}{ll}
A_{n}^{\langle 2\rangle }=A_{n}^{\langle 1\rangle }B_{n+1}+B_{n}^{\langle
1\rangle }A, & B_{n}^{\langle 2\rangle }=A_{n}^{\langle 1\rangle
}C_{n+1}+B_{n}^{\langle 1\rangle }B_{n}+C_{n}^{\langle 1\rangle }A, \\ 
C_{n}^{\langle 2\rangle }=B_{n}^{\langle 1\rangle }C_{n}+C_{n}^{\langle
1\rangle }B_{n-1}+D_{n}^{\langle 1\rangle }A, & D_{n}^{\langle 2\rangle
}=C_{n}^{\langle 1\rangle }C_{n-1}+D_{n}^{\langle 1\rangle }B_{n-2}.%
\end{array}
\label{[Sec5]-Matrix4TRR-coef}
\end{equation}%
Using the expressions stated above, the matrix coefficients (\ref%
{[Sec5]-Matrix4TRR-coef}) can be easily obtained as well. Beyond the
explicit expression of their respective entries, the key point is that they
are structured matrices, namely $A_{n}^{\langle 2\rangle }$ is
lower-triangular with one's in the main diagonal, $D_{n}^{\langle 2\rangle }$
is upper-triangular, and $B_{n}^{\langle 2\rangle }$, $C_{n}^{\langle
2\rangle }$ are full matrices. Therefore, the sequence of type~II vector
multiple orthogonal polynomials satisfy the following matrix four term
recurrence relation%
\begin{equation}
x\mathcal{P}_{n}=A_{n}^{\langle 2\rangle }\mathcal{P}_{n+1}+B_{n}^{\langle
2\rangle }\mathcal{P}_{n}+C_{n}^{\langle 2\rangle }\mathcal{P}%
_{n-1}+D_{n}^{\langle 2\rangle }\mathcal{P}_{n-2},\ \ n=2,3,\ldots \,,
\label{[Sec5]-Matrix4TRR-3}
\end{equation}%
Next,\ combining (\ref{[Sec5]-PnWnX}) with (\ref{[Sec5]-Matrix4TRR-3}), we
can assert that%
\begin{multline*}
x^{3}W_{n}(x^{3})%
\begin{bmatrix}
1 \\ 
x \\ 
x^{2}%
\end{bmatrix}%
=A_{n}^{\langle 2\rangle }W_{n+1}(x^{3})%
\begin{bmatrix}
1 \\ 
x \\ 
x^{2}%
\end{bmatrix}%
+B_{n}^{\langle 2\rangle }W_{n}(x^{3})%
\begin{bmatrix}
1 \\ 
x \\ 
x^{2}%
\end{bmatrix}
\\
+C_{n}^{\langle 2\rangle }W_{n-1}(x^{3})%
\begin{bmatrix}
1 \\ 
x \\ 
x^{2}%
\end{bmatrix}%
+D_{n}^{\langle 2\rangle }W_{n-2}(x^{3})%
\begin{bmatrix}
1 \\ 
x \\ 
x^{2}%
\end{bmatrix}%
,\ \ n=2,3,\ldots .
\end{multline*}%
The vector $%
\begin{bmatrix}
1 & x & x^{2}%
\end{bmatrix}%
^{T}$ can be removed, and after the shift $x^{3}\rightarrow x$ the above
expression may be simplified as%
\begin{equation}
xW_{n}(x)=A_{n}^{\langle 2\rangle }W_{n+1}(x)+B_{n}^{\langle 2\rangle
}W_{n}(x)+C_{n}^{\langle 2\rangle }W_{n-1}(x)+D_{n}^{\langle 2\rangle
}W_{n-2}(x),  \label{[Sec5]-Matrix4TRR-2}
\end{equation}%
where $W_{-1}=0_{3\times 3}$ and $W_{0}$ is certain constant matrix, for
every $n=1,2,\ldots $, which is the desired matrix four term recurrence
relation for $W_{n}(x)$.
\end{proof}

This kind of matrix high--term recurrence relation completely characterizes
certain type of orthogonality. Hence, we are going prove a Favard type
Theorem which states that, under the assumptions of Theorem \ref{[SEC5]-TH1}%
, the matrix polynomials $\{W_{n}\}$ are \textit{type~II matrix multiple
orthogonal} with respect to a system of two matrices of measures $\{d\mathbf{%
M}^{1},d\mathbf{M}^{2}\}$.

Next, we briefly review some of the standard facts on the theory of \textit{%
matrix orthogonality}, or \textit{orthogonality with respect to a matrix of
measures} (see \cite{D-CJM-95}, \cite{DG-JCAM-05} and the references
therein). Let $W,V\in \mathbb{P}^{3\times 3}$ be two matrix polynomials, and
let%
\begin{equation*}
\mathbf{M}(x)=%
\begin{bmatrix}
\mu _{11}(x) & \mu _{12}(x) & \mu _{13}(x) \\ 
\mu _{21}(x) & \mu _{22}(x) & \mu _{23}(x) \\ 
\mu _{31}(x) & \mu _{32}(x) & \mu _{33}(x)%
\end{bmatrix}%
,
\end{equation*}%
be a matrix with positive Borel measures $\mu _{i,j}(x)$. Let $\mathbf{M}(E)$
be positive definite for any Borel set $E\subset \mathbb{R}$, having finite
moments%
\begin{equation*}
\bar{\mu}_{k}=\int_{E}d\mathbf{M}(x)x^{k},\,\,\,k=0,1,2,\ldots
\end{equation*}%
of every order, and satisfying that%
\begin{equation*}
\int_{E}V(x)d\mathbf{M}(x)V^{\ast }(x),
\end{equation*}%
where $V^{\ast }\in \mathbb{P}^{3\times 3}$ is the adjoint matrix of $V\in 
\mathbb{P}^{3\times 3}$, is non-singular if the matrix leading coefficient
of the matrix polynomial $V$ is non-singular. Under these conditions, it is
possible to associate to a weight matrix $\mathbf{M}$, the Hermitian
sesquilinear form%
\begin{equation*}
\langle W,V\rangle =\int_{E}W(x)d\mathbf{M}(x)V^{\ast }(x).
\end{equation*}%
We then say that a sequence of matrix polynomials $\{W_{n}\}$, $W_{n}\in 
\mathbb{P}^{3\times 3}$ with degree $n$ and nonsingular leading coefficient,
is orthogonal with respect to $\mathbf{M}$\ if%
\begin{equation}
\langle W_{m},W_{n}\rangle =\Delta _{n}\delta _{m,n}\,,
\label{[Sec5]-Matrix-Orthog}
\end{equation}%
where $\Delta _{n}\in \mathcal{M}_{3\times 3}$ is a positive definite upper
triangular matrix, for $n\geq 0$.

We can define the \textit{matrix moment functional}$\ \mathrm{M}$ acting in $%
\mathbb{P}^{3\times 3}$ over $\mathcal{M}_{3\times 3}$,\ in terms of the
above matrix inner product, by $\mathrm{M}(WV)=\langle W,V\rangle $. This
construction is due to J\'odar \textit{et al.} (see \cite{JD-JD-97}, \cite%
{JDP-JAT-96}) where the authors extend to the matrix framework the linear
moment functional approach developed by Chihara in \cite{Chi78}. Hence, the
moments of $\mathbf{M}(x)$ and (\ref{[Sec5]-Matrix-Orthog}) can be written%
\begin{equation*}
\begin{array}{l}
\mathrm{M}(x^{k})=\bar{\mu}_{k}\,,\,\,\,k=0,1,2,\ldots , \\ 
\mathrm{M}(W_{m}V_{n})=\Delta _{n}\delta _{m,n}\,,\,\,\,m,n=0,1,2,\ldots .%
\end{array}%
\end{equation*}

Let $\boldsymbol{m}=(m_{1},m_{2})\in \mathbb{N}^{2}$ be a multi--index with
length $|\mathbf{m}|:=m_{1}+\cdots +m_{2}$ and let $\{\mathrm{M}^{1},\mathrm{%
M}^{2}\}$ be a set of matrix moment functionals as defined above. Let $\{W_{%
\mathbf{m}}\}$ be a sequence of matrix polynomials, with $\deg W_{\mathbf{m}%
} $ is at most $|\mathbf{m}|$. $\{W_{\mathbf{m}}\}$ is said to be a \textit{%
type~II} multiple orthogonal with respect to the set of linear functionals $%
\{\mathrm{M}^{1},\mathrm{M}^{2}\}$ and multi--index $\boldsymbol{m}$ if it
satisfy the following orthogonality conditions%
\begin{equation}
\mathrm{M}^{j}(x^{k}W_{\mathbf{n}})=0_{3\times 3}\,,\ \ k=0,1,\ldots
,n_{j}-1\,,\ \ j=1,2.  \label{[Sec5]-OrtConduj}
\end{equation}%
A multi--index $\boldsymbol{m}$ is said to be \textit{normal} for the set of
matrix moment functionals $\{\mathrm{M}^{1},\mathrm{M}^{2}\}$, if the degree
of $W_{\mathbf{m}}$ is exactly $|\mathbf{m}|=m$. Thus, in what follows we
will write $W_{\mathbf{m}}\equiv W_{|\mathbf{m}|}=W_{m}$.

In this framework, let consider the sequence of \textit{vector of matrix
polynomials} $\{\mathcal{B}_{n}\}$ where%
\begin{equation}
\mathcal{B}_{n}=%
\begin{bmatrix}
W_{2n} \\ 
W_{2n+1}%
\end{bmatrix}%
.  \label{[Sec5]-BnofWn}
\end{equation}%
We define the vector of matrix--functionals $\mathfrak{M}=%
\begin{bmatrix}
\mathrm{M}^{1} & \mathrm{M}^{2}%
\end{bmatrix}%
^{T}$, with $\mathfrak{M}:\mathbb{P}^{6\times 3}\rightarrow \mathcal{M}%
_{6\times 6}$, by means of the action $\mathfrak{M}$ on $\mathcal{B}_{n}$,
as follows%
\begin{equation}
\mathfrak{M}\left( \mathcal{B}_{n}\right) =%
\begin{bmatrix}
\mathrm{M}^{1}(W_{2n}) & \mathrm{M}^{2}(W_{2n}) \\ 
\mathrm{M}^{1}(W_{2n+1}) & \mathrm{M}^{2}(W_{2n+1})%
\end{bmatrix}%
\in \mathcal{M}_{6\times 6}\,.  \label{[Sec5]-UBn}
\end{equation}%
where%
\begin{equation}
\mathrm{M}^{i}(W_{j})=\int W_{j}(x)d\mathbf{M}^{i}=\Delta _{j}^{i},\quad
i=1,2,\text{ and }j=0,1,  \label{[Sec5]-tauiWj}
\end{equation}%
with $\{d\mathbf{M}^{1},d\mathbf{M}^{2}\}$ being a system of two matrix of
measures as described above. Thus, we say that a sequence of vectors of
matrix polynomials $\{\mathcal{B}_{n}\}$ is orthogonal with respect to a
vector of matrix functionals $\mathfrak{M}$ if%
\begin{equation}
\left. 
\begin{array}{rll}
i) & \mathfrak{M}(x^{k}\mathcal{B}_{n})=0_{6\times 6}\,, & k=0,1,\ldots
,n-1\,, \\ 
ii) & \mathfrak{M}(x^{n}\mathcal{B}_{n})=\Omega _{n}\,, & 
\end{array}%
\right\}  \label{[Sec5]-MatrixOrthoU}
\end{equation}%
where $\Omega _{n}$ is a regular block upper triangular $6\times 6$ matrix,
holds.

Now we are in a position to prove the following

\begin{teo}[Favard type]
\label{[SEC5]-TH2}Let $\{\mathcal{B}_{n}\}$ a sequence of vectors of matrix
polynomials of size $6\times 3$, defined in (\ref{[Sec5]-BnofWn}), with $%
W_{n}$ matrix polynomials satisfying the four term recurrence relation (\ref%
{[Sec5]-Matrix4TRR-2}). The following statements are equivalent:

\begin{itemize}
\item[$(a)$] The sequence $\{\mathcal{B}_{n}\}_{n\geq 0}$ is orthogonal with
respect to a certain vector of two matrix--functionals.

\item[$(b)$] There are sequences of scalar $6\times 6$ block matrices $%
\{A_{n}^{\langle 3\rangle }\}_{n\geq 0}$, $\{B_{n}^{\langle 3\rangle
}\}_{n\geq 0}$, and $\{C_{n}^{\langle 3\rangle }\}_{n\geq 0}$, with $%
C_{n}^{\langle 3\rangle }$\ block upper triangular non-singular matrix for $%
n\in \mathbb{N}$, such that the sequence $\{\mathcal{B}_{n}\}$ satisfy the
matrix three term recurrence relation%
\begin{equation}
x\mathcal{B}_{n}=A_{n}^{\langle 3\rangle }\mathcal{B}_{n+1}+B_{n}^{\langle
3\rangle }\mathcal{B}_{n}+C_{n}^{\langle 3\rangle }\mathcal{B}_{n-1},
\label{[Sec5]-MatrixFavard-1}
\end{equation}%
with $\mathcal{B}_{n-1}=%
\begin{bmatrix}
0_{3\times 3} & 0_{3\times 3}%
\end{bmatrix}%
^{T}$, $\mathcal{B}_{0}$ given, and $C_{n}^{\langle 3\rangle }$ non-singular.
\end{itemize}
\end{teo}

\begin{proof}
First we prove that $(a)$ implies $(b)$. Since the sequence of vector of
matrix polynomials $\{\mathcal{B}_{n}\}$ is a basis in the linear space $%
\mathbb{P}^{6\times 3}$, we can write%
\begin{equation*}
x\mathcal{B}_{n}=\sum_{k=0}^{n+1}\widetilde{A}_{k}^{n}\mathcal{B}%
_{k},\,\,\,\,\widetilde{A}_{k}^{n}\in \mathcal{M}_{6\times 6}\,.
\end{equation*}%
Then, from the orthogonality conditions (\ref{[Sec5]-MatrixOrthoU}), we get%
\begin{equation*}
\mathfrak{M}\left( x\mathcal{B}_{n}\right) =\mathfrak{M}(\widetilde{A}%
_{k}^{n}\mathcal{B}_{k})=0_{6\times 6},\quad k=0,1,\ldots ,\,n-2.
\end{equation*}%
Thus,%
\begin{equation*}
x\mathcal{B}_{n}=\widetilde{A}_{n+1}^{n}\mathcal{B}_{n+1}+\widetilde{A}%
_{n}^{n}\mathcal{B}_{n}+\widetilde{A}_{n-1}^{n}\mathcal{B}_{n-1}\,.
\end{equation*}%
Having $\widetilde{A}_{n+1}^{n}=A_{n}^{\langle 3\rangle }$, $\widetilde{A}%
_{n}^{n}=B_{n}^{\langle 3\rangle }$, and $\widetilde{A}_{n-1}^{n}=C_{n}^{%
\langle 3\rangle }$, the result follows.

To proof that $(b)$\ implies $(a)$, we know from Theorem \ref{[SEC5]-TH1}
that the sequence of vector polynomials $\{W_{n}\}$ satisfy a four term
recurrence relation with matrix coefficients. We can associate this matrix
four term recurrence relation (\ref{[Sec5]-Matrix4TRR-2}) with the block
matrix three term recurrence relation (\ref{[Sec5]-MatrixFavard-1}). Then,
it is sufficient to show that $\mathfrak{M}$ is uniquely determined by its
orthogonality conditions (\ref{[Sec5]-MatrixOrthoU}), in terms of the
sequence $\{C_{n}^{\langle 3\rangle }\}_{n\geq 0}$ in that (\ref%
{[Sec5]-MatrixFavard-1}).

Next, from (\ref{[Sec5]-BnofWn}) we can rewrite the matrix four term
recurrence relation (\ref{[Sec5]-Matrix4TRR-2}) into a matrix three term
recurrence relation%
\begin{equation*}
x\mathcal{B}_{n}=%
\begin{bmatrix}
0_{3\times 3} & 0_{3\times 3} \\ 
A_{2n+1}^{\langle 2\rangle } & 0_{3\times 3}%
\end{bmatrix}%
\mathcal{B}_{n+1}+%
\begin{bmatrix}
B_{2n}^{\langle 2\rangle } & A_{2n}^{\langle 2\rangle } \\ 
C_{2n+1}^{\langle 2\rangle } & B_{2n+1}^{\langle 2\rangle }%
\end{bmatrix}%
\mathcal{B}_{n}+%
\begin{bmatrix}
D_{2n}^{\langle 2\rangle } & C_{2n}^{\langle 2\rangle } \\ 
0_{3\times 3} & D_{2n+1}^{\langle 2\rangle }%
\end{bmatrix}%
\mathcal{B}_{n-1},
\end{equation*}%
where the size of $\mathcal{B}_{n}$\ is $6\times 3$. We give $\mathfrak{M}$
in terms of its block matrix moments, which in turn are given by the matrix
coefficients in (\ref{[Sec5]-MatrixFavard-1}). There is a unique vector
moment functional $\mathfrak{M}$ and hence two matrix measures $d\mathbf{M}%
^{1}$ and $d\mathbf{M}^{2}$, such that%
\begin{equation*}
\mathfrak{M}\left( \mathcal{B}_{0}\right) =%
\begin{bmatrix}
\mathrm{M}^{1}(W_{0}) & \mathrm{M}^{2}(W_{0}) \\ 
\mathrm{M}^{1}(W_{1}) & \mathrm{M}^{2}(W_{1})%
\end{bmatrix}%
=C_{0}^{\langle 3\rangle }\in \mathcal{M}_{6\times 6},
\end{equation*}%
where $\mathrm{M}^{i}(W_{j})$ was defined in (\ref{[Sec5]-tauiWj}). For the
first moment of $\mathfrak{M}_{0}$ we get%
\begin{equation*}
\mathfrak{M}_{0}=\mathfrak{M}\left( \mathcal{B}_{0}\right) =%
\begin{bmatrix}
\Delta _{0}^{1} & \Delta _{0}^{2} \\ 
\Delta _{1}^{1} & \Delta _{1}^{2}%
\end{bmatrix}%
=C_{0}^{\langle 3\rangle }.
\end{equation*}%
Hence, we have $\mathfrak{M}\left( \mathcal{B}_{1}\right) =0_{6\times 6}$,
which from (\ref{[Sec5]-MatrixFavard-1}) also implies $0_{6\times 6}=x%
\mathfrak{M}\left( \mathcal{B}_{0}\right) -B_{0}^{\langle 3\rangle }%
\mathfrak{M}\left( \mathcal{B}_{0}\right) =\mathfrak{M}_{1}-B_{0}^{\langle
3\rangle }\mathfrak{M}_{0}$. Therefore%
\begin{equation*}
\mathfrak{M}_{1}=B_{0}^{\langle 3\rangle }C_{0}^{\langle 3\rangle }.
\end{equation*}%
By a similar argument, we have%
\begin{eqnarray*}
0_{6\times 6} &=&\mathfrak{M}\left( (A_{0}^{\langle 3\rangle })^{-1}x^{2}%
\mathcal{B}_{0}-(A_{0}^{\langle 3\rangle })^{-1}B_{0}^{\langle 3\rangle }x%
\mathcal{B}_{0}-B_{1}^{\langle 3\rangle }(A_{0}^{\langle 3\rangle })^{-1}x%
\mathcal{B}_{0}+B_{1}(A_{0}^{\langle 3\rangle })^{-1}B_{0}^{\langle 3\rangle
}\mathcal{B}_{0}-C_{1}^{\langle 3\rangle }\mathcal{B}_{0}\right) \\
&=&(A_{0}^{\langle 3\rangle })^{-1}\mathfrak{M}_{2}-\left( (A_{0}^{\langle
3\rangle })^{-1}B_{0}^{\langle 3\rangle }+B_{1}^{\langle 3\rangle
}(A_{0}^{\langle 3\rangle })^{-1}\right) \mathfrak{M}_{1}+\left(
B_{1}^{\langle 3\rangle }(A_{0}^{\langle 3\rangle })^{-1}B_{0}^{\langle
3\rangle }-C_{1}^{\langle 3\rangle }\right) \mathfrak{M}_{0}\in \mathcal{M}%
_{6\times 6},
\end{eqnarray*}%
which in turn yields the second moment of $\mathfrak{M}$%
\begin{equation*}
\mathfrak{M}_{2}=\left( B_{0}^{\langle 3\rangle }+A_{0}^{\langle 3\rangle
}B_{1}^{\langle 3\rangle }(A_{0}^{\langle 3\rangle })^{-1}\right)
B_{0}^{\langle 3\rangle }C_{0}^{\langle 3\rangle }-A_{0}^{\langle 3\rangle
}\left( B_{1}^{\langle 3\rangle }(A_{0}^{\langle 3\rangle
})^{-1}B_{0}^{\langle 3\rangle }-C_{1}^{\langle 3\rangle }\right)
C_{0}^{\langle 3\rangle }.
\end{equation*}%
Repeated application of this inductive process, enables us to determine $%
\mathfrak{M}$ in a unique way through its moments, only in terms of the
sequences of matrix coefficients $\{A_{n}^{\langle 3\rangle }\}_{n\geq 0}$, $%
\{B_{n}^{\langle 3\rangle }\}_{n\geq 0}$ and $\{C_{n}^{\langle 3\rangle
}\}_{n\geq 0}$.

On the other hand, because of (\ref{[Sec5]-UBn}) and (\ref%
{[Sec5]-MatrixFavard-1}) we have%
\begin{equation*}
\mathfrak{M}\left( x\mathcal{B}_{n}\right) =0_{6\times 6},\quad n\geq 2.
\end{equation*}%
Multiplying by $x$ both sides of (\ref{[Sec5]-MatrixFavard-1}), from the
above result we get%
\begin{equation*}
\mathfrak{M}\left( x^{2}\mathcal{B}_{n}\right) =0_{6\times 6},\quad n\geq 3.
\end{equation*}%
The same conclusion can be drawn for $0<k<n$%
\begin{equation}
\mathfrak{M}\left( x^{k}\mathcal{B}_{n}\right) =0_{6\times 6},\quad 0<k<n,
\label{[Sec5]-Concl1}
\end{equation}%
and finally%
\begin{equation*}
\mathfrak{M}\left( x^{n}\mathcal{B}_{n}\right) =C_{n}^{\langle 3\rangle }%
\mathfrak{M}\left( x^{n-1}\mathcal{B}_{n-1}\right) .
\end{equation*}%
Notice that the repeated application of the above argument leads to%
\begin{equation}
\mathfrak{M}\left( x^{n}\mathcal{B}_{n}\right) =C_{n}^{\langle 3\rangle
}C_{n-1}^{\langle 3\rangle }C_{n-2}^{\langle 3\rangle }\cdots C_{1}^{\langle
3\rangle }C_{0}^{\langle 3\rangle }.  \label{[Sec5]-Concl2}
\end{equation}%
From (\ref{[Sec5]-Concl1}), (\ref{[Sec5]-Concl2}) we conclude%
\begin{eqnarray*}
\mathfrak{M}\left( x^{k}\mathcal{B}_{n}\right) &=&C_{n}^{\langle 3\rangle
}C_{n-1}^{\langle 3\rangle }C_{n-2}^{\langle 3\rangle }\cdots C_{1}^{\langle
3\rangle }C_{0}^{\langle 3\rangle }\delta _{n,k} \\
&=&\Omega _{n}\delta _{n,k},\quad n,k=0,1,\ldots ,\,k\leq n,
\end{eqnarray*}%
which are exactly the desired orthogonality conditions (\ref%
{[Sec5]-MatrixOrthoU}) for $\mathfrak{M}$ stated in the theorem.
\end{proof}


\end{document}